\documentclass[9pt, a4paper, centertags, openright]{amsart}

\usepackage{amsmath, amsfonts, amsthm, 
amssymb, amsxtra, bbm, enumerate, stmaryrd}
\usepackage[dvipsnames]{xcolor}
\usepackage[retainorgcmds]{IEEEtrantools}
\usepackage[mathscr]{euscript}
\usepackage{graphicx}
\usepackage[all, cmtip]{xy}
\xyoption{curve}

\usepackage{cite}

\usepackage{wasysym}
\usepackage[english]{babel}
\usepackage[utf8]{inputenc}
\usepackage{paralist}
\usepackage{textcomp}
\usepackage[colorinlistoftodos]{todonotes}
\usepackage{caption}

\usepackage{fancyhdr}
\usepackage[bottom]{footmisc}
\usepackage[hyphens]{url}

\definecolor{Myblue}{rgb}{0,0,0.6}
\usepackage[a4paper, colorlinks=true, 
citecolor=green, linkcolor=red, 
urlcolor=cyan, filecolor=black,
pdfpagemode=None, 
]{hyperref}

\usepackage{breakurl}

\calclayout
\makeatletter
\g@addto@macro{\thm@space@setup}{\thm@headfont{\bf}}

\theoremstyle{theorem}
\newtheorem{lem}{Lemma}[section]
\newtheorem{prop}[lem]{Proposition}
\newtheorem{cor}[lem]{Corollary}
\newtheorem{thm}[lem]{Theorem}

\newtheorem*{Thm1}{Theorem~1}

\newtheorem*{Cor}{Corollary}

\theoremstyle{remark}
\newtheorem{rem}[lem]{Remark}

\theoremstyle{definition}
\newtheorem{exa}[lem]{Example}
\newtheorem{exas}[lem]{Examples}
\newtheorem{defn}[lem]{Definition}
\newtheorem{quest}[lem]{Question}
\newtheorem{nn}[lem]{}

\numberwithin{equation}{section}

\newcommand{\Der}{\operatorname{Der}\nolimits}
\newcommand{\Inn}{\operatorname{Inn}\nolimits}

\newcommand{\add}{\operatorname{add}\nolimits}

\newcommand{\id}{\operatorname{id}\nolimits}
\newcommand{\Id}{\operatorname{Id}\nolimits}

\newcommand{\Mod}{\operatorname{\mathsf{Mod}}\nolimits}

\newcommand{\End}{\operatorname{End}\nolimits}
\newcommand{\Hom}{\operatorname{Hom}\nolimits}

\renewcommand{\Im}{\operatorname{Im}\nolimits}
\newcommand{\Ker}{\operatorname{Ker}\nolimits}

\newcommand{\Coker}{\operatorname{Coker}\nolimits}

\newcommand{\Ext}{\operatorname{Ext}\nolimits}

\newcommand{\Tor}{\operatorname{Tor}\nolimits}

\newcommand{\ev}{\mathrm{ev}}

\newcommand{\HH}{\operatorname{HH}\nolimits}
\newcommand{\OH}{\operatorname{H}\nolimits}

\newcommand{\Ob}{\operatorname{Ob}\nolimits}

\newcommand{\Out}{\operatorname{Out}\nolimits}

\newcommand{\abs}[1]{\ensuremath{\left\vert#1\right\vert}}

\def\A{{\mathsf A}}

\def\D{{\mathsf D}}

\def\E{{\mathsf E}}

\def\Z{{\mathbb Z}}

\def\op{\mathrm{op}}

\def\C{{\mathsf C}}

\def\bbm1{{\mathbbm 1}}

\newdir{ >}{{}*!/-5pt/@{>}}
\newdir{> }{{}*!/+5pt/@{>}}
\newdir{ <}{{}*!/-5pt/@{<}}
\newdir{>>> }{{}*!/+10pt/@{>}}

\entrymodifiers={+!!<0pt,\fontdimen22\textfont2>}

\begin{document}

\title[Hochschild cohomology and the $M$-relative center]{Exact sequences, Hochschild cohomology, and the Lie module structure over the $M$-relative center}
\author{Reiner Hermann}
\address{Reiner Hermann\\ Institutt for matematiske fag\\
NTNU\\ 7491 Trondheim\\ Norway}
\email{reiner.hermann@math.ntnu.no}
\thanks{}
\begin{abstract}  

In this article, we present actions by central elements on Hochschild cohomology groups with arbitrary bimodule coefficients, as well as an interpretation of these actions in terms of exact sequences. Since our construction utilises the monoidal structure that the category of bimodules possesses, we will further recognise that these actions are compatible with monoidal functors and thus, as a consequence, are invariant under Morita equivalences. By specialising the bimodule coefficients to the underlying algebra itself, our efforts in particular yield a description of the degree-$(n,0)$-part of the Lie bracket in Hochschild cohomology, and thereby close a gap in earlier work by S.\,Schwede.
\end{abstract}

\maketitle
\tableofcontents

%%%%%%%%%%%%%%%%%%%%%%%%%%%%%%%%
%%% Introduction
%%%%%%%%%%%%%%%%%%%%%%%%%%%%%%%%

\section{Introduction}

Let $A$ be an associative and unital algebra over a commutative ring $K$. The theory of Hochschild cohomology has, since it was introduced by G.\,Hochschild in 1945 (see \cite{Ho45}), developed into various areas of mathematics, e.g., algebraic geometry, as well as representation and deformation theory of associative algebras. One of its most intriguing features is probably the fact, that the \textit{Hochschild cohomology module} $\HH^\ast(A,A)$ with coefficients in $A$ carries the structure of a graded $K$-algebra. In 1963, M.\,Gerstenhaber showed that $\HH^\ast(A,A)$ actually is \textit{graded commutative} in that homogeneous elements of degrees $m, n \geqslant 0$ commute up to the sign $(-1)^{mn}$. In the very same article, Gerstenhaber provided a graded Lie bracket
$$
\{-,-\} : \HH^m(A,A) \times \HH^n(A,A) \longrightarrow \HH^{m+n-1}(A,A) \quad \text{(for $m , n \geqslant 0$)}
$$
of degree $-1$, being compatible with the multiplicative structure of $\HH^\ast(A,A)$ as it acts through graded derivations. For $A$ being projective over $K$, H.\,Cartan and S.\,Eilenberg interpreted $\HH^\ast(A,A)$ as the $\Ext$-algebra of $A$ in the category of bimodules over $A$ -- and thus pointed out that the \textit{Hochschild cohomology algebra} of $A$ may be comprehended in terms of arbitrary projective resolutions of the bimodule $A$, or, equivalently, arbitrary self-extensions of $A$, establishing understanding of its multiplicative structure through the Yoneda product (see \cite{CaEi56}, \cite{Yo54}). Gerstenhaber's Lie bracket on the other hand proved itself resistant of such a description for several decades.
\medskip

The shortcoming that Gerstenhaber's construction could only be grasped in terms of the Hochschild cocomplex, and by no means if starting with an arbitrary projective bimodule resolution of $A$, asked for significant improvement, as Gerstenhaber himself did explicitly, together with S.\,D.\,Schack (see \cite[p.\,256]{GeSch86}):
\begin{quote}
\it{What intrinsic reason is there for the existence of the graded Lie structure on $\HH^\ast(A,A)$? The cup product $($but not its graded com\-mutativity$)$ can be understood from the Yoneda theory; something is wanting to make the Lie structure equally transparent.}
\end{quote}
An answer to their prayers was given by S.\,Schwede in \cite{Sch98}, where he desribed Gerstenhaber's bracket in terms of bimodule self-extensions of $A$, utilizing the monoidal structure that the category of $A$-bimodules possesses. Schwede's construction has been generalised to exact monoidal categories in \cite{He14b}, in order to aquire a better understanding of functorial properties of the Lie bracket in Hochschild cohomology. However, Schwede's interpretation misses a significant piece of the picture, as it does not describe the brackets
$$
\{-,-\} : \HH^n(A,A) \times \HH^0(A,A) \longrightarrow \HH^{n-1}(A,A) \quad \text{(for $n \geqslant 0$)}.
$$
These maps are, in general, non-trivial (see for instance \cite{LeZh13}) and thus cannot be ignored. The primary goal of the present article is to close this gap.
\medskip

Let $A$ be a unital and associative algebra over a commutative ring $K$. We will write $\otimes$ for $\otimes_K$, and denote by $A^\ev = A \otimes A^\op$ the \textit{enveloping algebra} of $A$. Fix an $A^\ev$-module $M$. We define the \textit{$M$-relative center} of $A$ to be $Z_M(A) = \{ z \in Z(A) \mid zm = mz \text{ for all $m \in M$}\}$. Clearly, $Z_A(A) = Z(A) = \HH^0(A,A)$. By imitating Gerstenhaber's construction on the Hochschild cocomplex, we obtain a map
$$
[-,-]_M : \HH^n(A,M) \times Z_M(A) \longrightarrow \HH^{n-1}(A,M) \quad \text{(for $n \geqslant 0$)},
$$
which recovers the Lie bracket $\{-,-\}$ in degrees $(n,0)$ for $M = A$. We are going to provide an interpretation of this map by $n$-extensions of $A$ by $M$. Let us give some idea of the construction, whose details will be presented in Sections \ref{sec:funda}--\ref{sec:actmonoidal}.
\medskip

For a ring $R$ and $R$-modules $U$ and $V$, we denote by $\mathcal Ext^n_R(U,V)$ the \textit{category of $n$-extensions of $U$ by $V$}. Due to V.\,Retakh, see \cite{Re86}, there is an isomorphism
$$
\pi_i\mathcal Ext^n_R(U,V) \longrightarrow \Ext^{n-i}_R(U,V)
$$
for every $n \geqslant 0$, and $0 \leqslant i \leqslant n$. Therefore, in particular, the fundamental group $\pi_1(\mathcal Ext^n_R(U,V),S)$ of $\mathcal Ext^n_R(U,V)$ based at some extension $S$ is isomorphic to $\Ext^{n-1}_R(U,V)$. By an interpretation of D.\,Quillen, cf.~\cite{Qu72}, $\pi_1(\mathcal Ext^n_R(U,V),S)$ can be thought of as equivalence classes of \textit{loops} at $S$ in the category $\mathcal Ext^n_R(U,V)$.
\medskip

Given a map $f \in \Hom_{A^\ev}(A,A)$, we obtain two homomorphisms $f_\lambda^M, f_\varrho^M : M \rightarrow M$ by taking the unit isomorphisms $\lambda_M : A \otimes_A M \rightarrow M$ and $\varrho_M : M \otimes_A A \rightarrow M$ into account:
$$
f_\lambda^M = \lambda_M \circ (f \otimes_A M) \circ \lambda^{-1}_M, \quad f_\varrho^M = \varrho_M \circ (M \otimes_A f) \circ \varrho^{-1}_M.
$$
Essentially, the maps $f_\lambda^M$ and $f_\varrho^M$ are given by left and right multiplication with $f(1)$; they will, in general, not agree.

Since the homomorphisms $f_\lambda^M$ and $f_\varrho^M$ are natural in $M$, they give rise to endomorphisms $F_\lambda, F_\varrho : S \rightarrow S$ of \textit{complexes} (but a priori not of extensions) for every extension $S \in \mathcal Ext^n_{A^\ev}(A,M)$. Assuming $f_\lambda^M = f_\varrho^M$, we will construct a loop
\begin{equation}\label{eq:loop}\tag{$\star$}
S \# f \longrightarrow f \# S \longleftarrow S \# f
\end{equation}
in $\mathcal Ext^{n}_{A^\ev}(A,M)$, were $\#$ denotes the Yoneda product on extensions. Since the groups $\pi_1(\mathcal Ext^{n}_{A^\ev}(A,M), S \# f)$ and $\Ext^{n-1}_{A^\ev}(A,M)$ identify by Retakh's theorem, the loop (\ref{eq:loop}) defines the equivalence class of an $(n-1)$-extension of $A$ by $M$. If we denote by $Z_M(\Mod(A^\ev))$ the set of all morphisms $f \in \Hom_{A^\ev}(A,A)$ whose \textit{$M$-relative defect} $\Delta^M(f) = f^M_\lambda - f^M_\varrho$ vanishes, we arrive at a map
$$
\langle -,-\rangle_M : \Ext^n_{A^\ev}(A,M) \times Z_M(\Mod(A^\ev)) \longrightarrow \Ext^{n-1}_{A^\ev}(A,M) \quad \text{(for $n \geqslant 0$)}.
$$
Observe that $Z_M(\Mod(A^\ev)) \cong Z_M(A)$ by sending $f$ to $f(1)$. We can now state our main result.

\begin{Thm1}[= Thm.\,\ref{thm:mainthm}]
The following diagram commutes for $n = 0,1$.
$$
\xymatrix@C=30pt{
\HH^n(A,M) \times Z_M(A) \ar[r]^-{[-,-]_M} \ar[d]_-{\cong} & \HH^{n-1}(A,M) \ar[d]^-{\cong}\\
\Ext^n_{A^\ev}(A,M) \times Z_{M}(\Mod(A^\ev)) \ar[r]^-{\langle -, -\rangle_M} & \Ext^{n-1}_{A^\ev}(A,M)
}
$$
It also commutes for $n > 1$ provided that $A$ is projective as a $K$-module.
\end{Thm1}

By taking our construction under deeper analysis in the more general context of exact monoidal categories, the main theorem leads to the following interplay of the maps $[-,-]_M$ with braidings on the monoidal category $(\Mod(Z(A)^\ev), \otimes_{Z(A)}, Z(A))$.

\begin{Cor}[$\subseteq$ Cor.\,\ref{cor:funccompat} and Cor.\,\ref{cor:vanish}]
Let $A$ be projective as a $K$-module. Consider the following statements.
\begin{enumerate}[\rm(1)]
\item\label{intcor:1} $\HH^0(Z(A),M) = M$ for all $M \in \Mod(A^\ev)$.
\item\label{intcor:2} $Z(A) = Z_{M}(A)$ for all $M \in \Mod(A^\ev)$.
\item\label{intcor:3} $Z(A) = Z_{A \otimes A}(A)$.
\item\label{intcor:4} $(\Mod(Z(A)^\ev), \otimes_{Z(A)}, Z(A))$ is braided monoidal.
\item\label{intcor:5} $[-,-]_M$ vanishes for all $M \in \Mod(A^\ev)$.
\end{enumerate}
Then one has the implications
$$
(\ref{intcor:1}) \ \Longleftrightarrow \ (\ref{intcor:2}) \ \Longleftrightarrow \ (\ref{intcor:3}) \ \Longleftrightarrow \ (\ref{intcor:4}) \ \Longrightarrow \ (\ref{intcor:5}) \,.
$$
Moreover, the map $[-,-]_M$ is invariant under Morita equivalences.
\end{Cor}

Furthermore, in the commutative case, we acquire a description of the kernel of a given derivation $A \rightarrow A$ in terms of some relative center.

\begin{Cor}[see Cor.\,\ref{cor:Outdervanish} and Rem.\,\ref{rem:kerder}]
Let $D : A \rightarrow M$ be a $K$-linear derivation for some $A^\ev$-module $M$. Then $\Ker(D) \cap Z_M(A) = \Ker(D\mathord{\upharpoonright}_{Z_M(A)}) = Z_{E_D}(A)$, where $E_D$ denotes the $A^\ev$-module
%$\left.D\right|_{Z_M(A)}$
$$
E_D = \frac{(A \otimes A) \oplus M}{\{(a \otimes b - ab \otimes 1, D(a)b) \mid a,b \in A\}} \, .
$$
\end{Cor}

\medskip

This article is organised as follows. In Section \ref{sec:exactmono}, we will built up the necessary foundations on exact and monoidal categories. Afterwards, in Section \ref{sec:funda}, we turn ourselves to fundamental groups of categories. We will present an explicit description of the Retakh isomorphism for fundamental groups of extension categories over exact categories, and explain in detail an even more explicit version for module categories going back to Schwede. Section \ref{sec:Mrel} will give a recap on the theory of Hochschild cohomology, and the definition of the Lie bracket associated with it. It will further introduce the $M$-relative center, along with its action on $\HH^\ast(A,M)$. In Section \ref{sec:actmonoidal} we will describe how to interpret this action in terms of extensions. Finally, in Section \ref{sec:proofs}, we will provide the proofs for our main results.

%%%%%%%%%%%%%%%%%%%%%%%%%%%%%%%%
%%% Prerequisites on exact and monoidal categories
%%%%%%%%%%%%%%%%%%%%%%%%%%%%%%%%

\section{Prerequisites on exact and monoidal categories}\label{sec:exactmono}

\begin{nn}
Let us recall the notions of exact and monoidal categories and structure preserving functors (exact and monoidal functors) between them. For further details on exact categories, we refer to \cite{Ke90} and \cite{Qu72}, whereas the textbooks \cite{AgMa10} and \cite{MaL98} provide background material on monoidal categories. In the following section, and in fact for the entire article, we fix a commutative ring $K$.
\end{nn}

\begin{nn}
An \textit{exact} $K$-category is a pair $(\C, i_\C)$ consisting of an additive $K$-category $\C$ and a full and faithful embedding $i_\C: \C \rightarrow \A_\C$ into an abelian $K$-category $\A_\C$, such that the essential image
$$
\Im(i_\C) = i_\C \C = \{ A \in \A_\C \mid i_\C(C) \cong A \text{ for some $C \in \C$}\}
$$
of $i_\C$ is an extension closed subcategory of $\A_\C$. Exact categories admit a sensible notion of exact sequences. Let $\C = (\C, i_\C)$ be an exact $K$-category. A sequence $0 \rightarrow C'' \rightarrow C \rightarrow C' \rightarrow 0$ is an \textit{admissible short exact sequence in $\C$}, if its image $0 \rightarrow i_\C(C'') \rightarrow i_\C(C) \rightarrow i_\C(C') \rightarrow 0$ under $i_\C$ is exact in $\A_\C$. Given such an admissible short exact sequence $0 \rightarrow C'' \rightarrow C \rightarrow C' \rightarrow 0$, the morphism $C'' \rightarrow C$ is called an \textit{admissible monomorphism}, whereas $C \rightarrow C'$ is an \textit{admissible epimorphism}. The class of admissible short exact sequences is closed under taking isomorphisms and direct sums (in the category of chain complexes over $\C$).
\end{nn}

\begin{nn}
Each exact $K$-category $(\C, i_\C)$ is closed under taking pushouts along admissible monomorphisms and pullbacks along admissible epimorphisms. Assume that $(\D, i_\D)$ is another exact $K$-category, and let $\mathfrak X: \C \rightarrow \D$ be an \textit{exact} $K$-linear functor, that is, it takes admissible short exact sequences in $\C$ to admissible short exact sequences in $\D$. Each such functor preserves pushouts along admissible monomorphisms and pullbacks along admissible epimorphisms.
\end{nn}

\begin{nn}
For an integer $n \geqslant 1$, a sequence
$$
S \quad \equiv \quad 0 \longrightarrow C'' \longrightarrow C_{n-1} \longrightarrow \cdots \longrightarrow C_0 \longrightarrow C' \longrightarrow 0
$$
of morphisms in an exact $K$-category $(\C, i_\C)$ is called an \textit{admissible $n$-extension} (\textit{of $C'$ by $C''$}) in case $i_\C S$ is exact in $\A_\C$, and $\Ker(i_\C(C_0 \rightarrow C'))$ and $\Ker(i_\C(C_k \rightarrow C_{k-1}))$ belong to $i_\C \C$ for all $k = 1, \dots, n-1$. 
\end{nn}

Let us turn to monoidal categories.

\begin{nn}
Recall that a \textit{monoidal category} is a 6-tuple $(\mathsf C, \otimes, \mathbbm 1, \alpha, \lambda, \varrho)$, where $\mathsf \C$ is a category, $\otimes: \C \times \C \rightarrow \C$ is a functor, $\mathbbm 1$ is an object in $\C$, and
\begin{gather*}
\alpha: - \otimes (- \otimes-) \longrightarrow (- \otimes -) \otimes - \ , \\
\lambda: \mathbbm 1 \otimes - \longrightarrow \Id_{\C} \ , \\
\varrho: - \otimes \mathbbm 1 \longrightarrow \Id_{\C}
\end{gather*}
are isomorphisms of functors such that, for all objects $W, X,Y,Z$ in $\C$,
$$
\xymatrix@C=40pt{
W \otimes (X \otimes (Y \otimes Z)) \ar[r]^-{\alpha_{W,X,Y \otimes Z}} \ar[d]_-{W \otimes \alpha_{X,Y,Z}} & (W\otimes X) \otimes (Y \otimes Z) \ar[r]^{\alpha_{W \otimes X, Y, Z}} & ((W \otimes X) \otimes Y) \otimes Z \ \ \\
W \otimes ((X \otimes Y) \otimes Z) \ar[rr]^-{\alpha_{W,X \otimes Y,Z}} & & (W \otimes (X \otimes Y)) \otimes Z \ar[u]_{\alpha_{W,X,Y} \otimes Z} \ ,
}
$$
commutes and $(\varrho_X \otimes Y) \circ \alpha_{X,\mathbbm 1,Y} = X \otimes \lambda_Y$. In this situation, $\otimes$ is a \textit{monoidal} (or \textit{tensor}) \textit{product functor for $\C$} and $\mathbbm 1$ is the \textit{$($tensor$)$ unit} of $\otimes$.

The monoidal category $(\C, \otimes, \mathbbm 1, \alpha, \lambda, \varrho)$ is a \textit{braided monoidal category} provided that there are natural isomorphisms $\gamma_{X,Y}: X \otimes Y \rightarrow Y \otimes X$ (for  $X, Y \in \Ob \C$) such that the diagrams
$$
\xymatrix@C=40pt{
X \otimes (Y \otimes Z) \ar[r]^-{\alpha_{X,Y,Z}} \ar[d]_-{X \otimes \gamma_{Y,Z}} & (X\otimes Y) \otimes Z \ar[r]^{\gamma_{X \otimes Y, Z}} & Z \otimes (X \otimes Y) \ar@<-3pt>[d]^{\alpha_{Z,X,Y}} \ \ \\
X \otimes (Z \otimes Y) \ar[r]^-{\alpha_{X,Z,Y}} & (X \otimes Z) \otimes Y \ar[r]^-{\gamma_{X,Z} \otimes Y} & (Z \otimes X) \otimes Y \ ,
}
$$
and
$$
\xymatrix@C=40pt{
(X \otimes Y) \otimes Z \ar[r]^-{\alpha_{X,Y,Z}^{-1}} \ar[d]_-{\gamma_{X,Y} \otimes Z} & X \otimes (Y \otimes Z) \ar[r]^-{\gamma_{X, Y \otimes Z}} & (Y \otimes Z) \otimes X \ar[d]^-{\alpha_{Y,Z,X}^{-1}} \\
(Y \otimes X) \otimes Z \ar[r]^-{\alpha_{Y,X,Z}^{-1}} & Y \otimes (X \otimes Z) \ar[r]^-{Y \otimes \gamma_{X, Z}} & Y \otimes (Z \otimes X) 
}
$$
commute for all $X,Y,Z \in \Ob \C$. In this case, $\gamma$ is a \textit{braiding} on the monoidal category $\C$.
If further $(\gamma_{X,Y})^{-1} = \gamma_{Y,X}$ for all $X, Y \in \Ob \C$, we say that the monoidal category is \textit{symmetric} and that $\gamma$ is a \textit{symmetry} on it.
\end{nn}
\begin{rem}\label{rem:opposite} Let $(\C, \otimes, \mathbbm 1, \alpha, \lambda, \varrho)$ be a monoidal category.
\begin{enumerate}[\rm(1)]
\item We will often suppress a huge part of the structure morphisms and simply write $(\C, \otimes, \mathbbm 1)$ instead of $(\C, \otimes, \mathbbm 1, \alpha, \lambda, \varrho)$; if they are needed without priorly having been mentioned, we will refer to them as $\alpha_\C$, $\lambda_\C$ and $\varrho_\C$.
\item It follows from the axioms (cf. \cite[Prop.\,1.1]{JoSt93}) that the following equations hold true for all $X, Y, Z \in \Ob \C$:
$$
\lambda_\bbm1 = \varrho_\bbm1, \quad \varrho_{X \otimes Y} \circ \alpha_{X,Y,\bbm1} = X \otimes \varrho_Y, \quad (\lambda_X \otimes Y) \circ \alpha_{\bbm1, X, Y} = \lambda_{X \otimes Y} \, .
$$
\item Let $\gamma$ be a braiding on $(\C, \otimes, \mathbbm 1, \alpha, \lambda, \varrho)$. Then $\varrho_X \circ \gamma_{\bbm1, X} = \lambda_X$ and $\lambda_X \circ \gamma_{X, \bbm1} = \varrho_X$ for all $X \in \Ob \C$ (cf. \cite[Prop.\,2.1]{JoSt93}).
\item Note that if $(\mathsf C, \otimes, \mathbbm 1, \alpha, \lambda, \varrho)$ is a monoidal category (with braiding $\gamma$), then so is $\mathsf C^\op$ together with the structure morphisms $\alpha^{-1}$, $\lambda^{-1}$ and $\varrho^{-1}$ (with braiding $\gamma^{-1}$).
\end{enumerate}
\end{rem}
\begin{nn}
We say that a monoidal category $(\C, \otimes, \mathbbm 1)$ is a \textit{tensor $K$-category}, if $\C$ is $K$-linear and the tensor product functor $\otimes: \C \times \C \rightarrow \C$ is $K$-bilinear on morphisms, that is, it factors through the \textit{tensor product category} $\C \otimes_K \C$ which is defined as follows:
\begin{align*}
\Ob(\C \otimes_K \C) &:= \Ob(\C \times \C),\\
\Hom_{\C \otimes_K \C}(\underline{X}, \underline{Y}) &:= \Hom_\C(X_1,Y_1) \otimes_K \Hom_\C(X_2, Y_2)
\end{align*}
for objects $\underline{X} = (X_1, X_2)$ and $\underline{Y} = (Y_1, Y_2)$ in $\C \times \C$. A tensor $K$-category is \textit{braided} (\textit{symmetric}) if its underlying monoidal category is braided (symmetric).
\end{nn}
\begin{nn}\label{def:monoidalfunc}
We are going to recall the definition of certain structure preserving functors between monoidal categories. Let $(\mathsf C, \otimes_{\mathsf C}, \mathbbm 1_{\mathsf C})$ and $(\mathsf D, \otimes_{\mathsf D}, \mathbbm 1_{\mathsf D})$ be monoidal categories. Let $\mathfrak A: \mathsf C \rightarrow \mathsf D$ be a functor, and
\begin{align*}
& \phi_{X,Y}: \mathfrak A X \otimes_{\mathsf D} \mathfrak A Y \longrightarrow \mathfrak A (X \otimes_{\mathsf C} Y),\\
& \psi_{X,Y}: \mathfrak A (X \otimes_{\mathsf C} Y) \longrightarrow \mathfrak A X \otimes_{\mathsf D} \mathfrak A Y,
\end{align*}
be natural morphisms in $\D$ (for $X,Y \in \Ob\mathsf C$). Further, let $\phi_0: \mathbbm 1_{\mathsf D} \rightarrow \mathfrak A \mathbbm 1_{\mathsf C}$ and $\psi_0: \mathfrak A \mathbbm 1_{\mathsf C} \rightarrow \mathbbm 1_{\mathsf D}$ be morphisms in $\mathsf D$.

The triple $(\mathfrak A, \phi, \phi_0)$ is called an \textit{almost strong monoidal functor} if $\phi_0$ is invertible and the following diagrams commute for all $X,Y,Z \in \Ob \mathsf C$.

\begin{equation*}
\xymatrix@C=40pt{
\mathfrak A X \otimes_\D (\mathfrak A Y \otimes_\D \mathfrak A Z) \ar[r]^{\mathfrak A X \otimes_\D \phi_{Y,Z}} \ar[d]_{\alpha_\D\mathfrak A} & \mathfrak A X
\otimes_\D \mathfrak A(Y \otimes_\C Z) \ar[r]^{\phi_{X,Y \otimes_\C Z}} & \mathfrak A(X \otimes_\C (Y \otimes_\C Z)) \ar[d]^{\mathfrak A \alpha_\C}\\
(\mathfrak A X \otimes_\D \mathfrak A Y) \otimes_\D \mathfrak A Z \ar[r]^-{\phi_{X,Y} \otimes_\D \mathfrak A Z} & \mathfrak A (X \otimes_\C Y) \otimes_\D \mathfrak A Z
\ar[r]^-{\phi_{X \otimes_\C Y,Z}} & \mathfrak A ((X \otimes_\C Y) \otimes_\C Z)
}
\end{equation*}
\begin{equation*}
\xymatrix{
\mathbbm 1_\D \otimes_\C \mathfrak A X  \ar[d]_{\phi_0 \otimes_\D \mathfrak A X} \ar[r]^-{\lambda_{\D}\mathfrak A} & \mathfrak A X  \\
\mathfrak A \mathbbm 1_\C \otimes_\D \mathfrak A X \ar[r]^-{\phi_{\mathbbm 1_\C, X}} & \mathfrak A(\mathbbm 1_\C \otimes_\C X) \ar[u]_{\mathfrak A \lambda_{\C}}
}
\quad
\xymatrix{
\mathfrak A X \otimes_\D  \mathbbm 1_\D \ar[d]_{\mathfrak A X \otimes_\D \phi_0} \ar[r]^-{\varrho_{\D}\mathfrak A} & \mathfrak A X  \\
\mathfrak A X \otimes_\D \mathfrak A \mathbbm 1_\C \ar[r]^-{\phi_{X, \mathbbm 1_\C}} & \mathfrak A(X \otimes_\C \mathbbm 1_\C)
\ar[u]_{\mathfrak A \varrho_{\C}}
}
\end{equation*}
The triple $(\mathfrak A, \psi, \psi_0)$ is called an \textit{almost costrong monoidal functor} if $(\mathfrak A^\op, \psi, \psi_0)$ is an almost strong monoidal functor. The triple $(\mathfrak A, \phi, \phi_0)$ is called a \textit{strong monoidal functor} if it is an almost strong monoidal functor and $\phi$ is invertible. The triple $(\mathfrak A, \psi, \psi_0)$ is called a \textit{costrong monoidal functor} if it is an almost costrong monoidal functor and $\psi$ is invertible.
\end{nn}

\begin{exas}\label{exa:monoidal}
\begin{enumerate}[\rm(1)]
\item If $(\C, \otimes, \bbm1)$ is a tensor $K$-category, the additive closure $\add(\bbm1)$ of the unit $\bbm1$ (i.e., the full subcategory of objects in $\C$ being isomorphic to some direct summand of a finite direct sum of copies of $\bbm1$) is a braided monoidal subcategory of $\C$.
\item\label{exa:monoidal:1} In the following, we will be mainly interested in the monoidal category of $K$-symmetric $A$-bimodules. It can be realised as the category of left modules over the \textit{enveloping algebra} $A^\ev = A \otimes A^\op$. The monoidal category $(\Mod(A^\ev), \otimes_A, A)$ is braided if, and only if, there is an invertible element $\mathbf r = \mathbf r_1 \otimes \mathbf r_2 \otimes \mathbf r_2 \in A \otimes A \otimes A$ (where implicit summation is understood) such that, for all $a \in A$,
\begin{align*}
\mathbf r_1 \otimes a\mathbf r_2 \otimes \mathbf r_3 &= \mathbf r_1 \otimes \mathbf r_2 \otimes \mathbf r_3a ,\\
\mathbf r_1\mathbf r_2 \otimes \mathbf r_3 &= 1 \otimes 1 , \\
\mathbf r_2 \otimes \mathbf r_3\mathbf r_1 &= 1 \otimes 1 .
\end{align*}
We refer to \cite{AgCaMi12} for a detailed analysis of braidings on the category of $A^\ev$-modules, especially \cite[Thm.\,3.1]{AgCaMi12} and \cite[Thm.\,3.2]{AgCaMi12}.
\item Let $\Gamma$ be a bialgebra, with comultiplication $\Delta : \Gamma \rightarrow \Gamma \otimes \Gamma$ and counit $\varepsilon : \Gamma \rightarrow K$. Then $K$ is a $\Gamma$-module through $\varepsilon$ and $(\Mod(\Gamma), \otimes, K)$ becomes a monoidal category. For two $\Gamma$-modules $M$ and $N$ the $\Gamma$-module structure on $M \otimes N$ is given by
$$
\gamma (m \otimes n) = \Delta(\gamma) \cdot (m \otimes n) \quad \text{(for $m \in M$, $n \in N$)}.
$$
The monoidal category $(\Mod(\Gamma), \otimes, K)$ is symmetric, if $\Gamma$ is cocommutative.
\end{enumerate}
\end{exas}

\begin{exa}\label{exa:monoidalfunc}
Coming back to Example \ref{exa:monoidal}(\ref{exa:monoidal:1}), recall that the algebra $A$ is \textit{Morita equivalent} to a $K$-algebra $B$, if there is a progenerator $P$ for $A$ (that is, a finitely generated projective $A$-module $P$ such that $P \cong A \oplus Q$ for some $A$-module $Q$) with $B \cong \End_A(P)^\op$. The functor $\Hom_A(P,-) : \Mod(A) \rightarrow \Mod(B)$ will be an equivalence then. The progenerator $P$ gives rise to a progenerator $P^\ev$ for $A^\ev$, namely, set $P^\ev = P \otimes \Hom_A(P,A)$. The opposite endomorphism ring of $P^\ev$ over $A^\ev$ is isomorphic to $B^\ev$, hence $A^\ev$ and $B^\ev$ are Morita equivalent. The equivalence $\Mod(A^\ev) \xrightarrow{\sim} \Mod(B^\ev)$ defined by $\Hom_{A^\ev}(P^\ev, -)$ is an \textit{almost strong} monoidal functor; see \cite[Sec.\,5.4]{He14b} for details.
\end{exa}

%%%%%%%%%%%%%%%%%%%%%%%%%%%%%%%%
%%% Fundamental groups
%%%%%%%%%%%%%%%%%%%%%%%%%%%%%%%%

\section{Fundamental groups}\label{sec:funda}

\begin{nn}
Let $\C$ be a category and let $X$ be an object in $\C$. Recall that $\C$ is a \textit{groupoid} if every morphism in $\C$ is invertible. The \textit{fundamental group} $\pi_1(\C,X)$ of $\C$ at the \textit{base point} $X$ is given by the fundamental group at $X$ of the geometric realisation of the nerve of $\C$. Alternatively, $\pi_1(\C,X)$ may be expressed as $\End_{\mathsf G(\C)}(X)$, where $\mathsf G(\C)$ denotes the so called \textit{fundamental} or \textit{Quillen groupoid} of $\C$ (see \cite{Qu72}). It comes with a functor $g_\C : \C \rightarrow \mathsf G(\C)$ which is universal in the following sense: for every groupoid $\mathsf G$ and every functor $\C \rightarrow \mathsf G$ there is a unique functor $\mathsf G(\C) \rightarrow \mathsf G$ such that the diagram
$$
\xymatrix@!C=10pt@!R=10pt{
& \C \ar[dr] \ar[dl]_-{g_\C} & \\
\mathsf G(\C) \ar[rr] && \mathsf G
}
$$
commutes. In particular, for any functor $\mathfrak X: \C \rightarrow \D$, we get a unique functor $\mathsf G(\mathfrak X): \mathsf G(\C) \rightarrow \mathsf G(\D)$ such that $\mathsf G(\mathfrak X) \circ g_\C = g_\D \circ \mathfrak X$. If $\C$ is a groupoid, then $\mathsf G(\C) \cong \C$ as categories.

Let us construct $\mathsf G(\C)$ explicitly. A \textit{path} from $X$ to $Y$ in $\C$ is a sequence of objects $X = X_0, X_1, \dots, X_n = Y$ and morphisms $f_0, \dots, f_{n-1}$ such that
$$
f_i \in \Hom_\C(X_i, X_{i+1}) \quad \text{or} \quad f_i \in \Hom_\C(X_{i+1}, X_{i}) \quad \text{(for $i = 0, \dots, n-1$).}
$$
We denote such a path by $w = (f_0, \dots, f_{n-1})$. The number $n$ is the \textit{length} of the path $w$ and a path from $X$ to $X$ is a \textit{loop} based at $X$. There is a unique loop of length $0$ based at $X$. Two paths $w = (f_0, \dots, f_{n-1})$ and $w' = (f'_0, \dots, f'_{n})$ from $X$ to $Y$ are \textit{elementary homotopic} if $w'$ arises from $w$ by replacing a morphism in $w$ which fits inside a commutative triangle
$$
\xymatrix@!C=10pt@R=6pt{
&&& U \ar[ddr] \ar[ddl] & \\
\Delta & \equiv &&&\\
&&V \ar[rr] & & W
}
$$
by the other two morphisms in $\Delta$. We further require that the loop of the length $0$ at $X$ is elementary homotopic to the loop of length $1$, given by the identity of $X$. 
\medskip

Denote by $\mathrm{Path}_\C(X,Y)$ the set of all paths from $X$ to $Y$ in $\C$. Now, the objects of $\mathsf G(\C)$ are given by the objects of $\C$; the set $\Hom_{\mathsf G(\C)}(X,Y)$ of morphisms $X \rightarrow Y$ in $\mathsf G(\C)$ is given by the quotient
$$
\Hom_{\mathsf G(\C)}(X,Y) = \mathrm{Path}_\C(X,Y) / \sim \, ,
$$
where $\sim$ denotes equivalence relation on $\mathrm{Path}_\C(X,Y)$ generated by elementary homotopy. The functor $g_\C$ is given by the identity on objects, whereas it sends a morphism $f$ in $\C$ to the equivalence class of the path $w = (f)$. If $\mathfrak X : \C \rightarrow \mathsf G$ is a functor into a groupoid $\mathsf G$, the unique functor $\overline{\mathfrak X}:\mathsf G(\C) \rightarrow \mathsf G$ with $\overline{\mathfrak X} \circ g_\C = \mathfrak X$ is given by sending (the equivalence class of) a path $(f_0, \dots, f_{n-1})$ in $\Hom_{\mathsf G(\C)}(X,Y)$ to
$$
\mathfrak X(f_{n-1})^{\sigma_{n-1}} \circ \cdots \circ \mathfrak X(f_0)^{\sigma_0}.
$$
The exponent $\sigma_i$ for $i = 0, \dots, n-1$ is defined to be $t_i - s_i \in \{-1, 1\}$ when $f_i: X_{s_i} \rightarrow X_{t_i}$, $s_i, t_i \in \{i, i+1\}$. If $\mathsf G = \C$ is itself a groupoid and $\mathfrak X = \Id_\C$, then $\overline{\mathfrak X}$ is an equivalence of categories.
\end{nn}

\begin{nn}\label{nn:uC}
Let $\C$ be an exact $K$-category and $n \geqslant 1$ be an integer. For objects $X$ and $Y$ in $\C$, a morphism $f: S \rightarrow T$ of admissible $n$-extensions $S$ and $T$ of $X$ by $Y$ is a commutative diagram
$$
\xymatrix@C=18pt{
S \ar[d]_-f & \equiv & 0 \ar[r] & Y \ar[r] \ar@{=}[d] & E_{n-1} \ar[d]_-{f_{n-1}} \ar[r] & \cdots \ar[r] & E_0 \ar[r] \ar[d]^-{f_0} & X \ar[r] \ar@{=}[d] & 0\\
T & \equiv & 0 \ar[r] & Y \ar[r] & F_{n-1} \ar[r] & \cdots \ar[r] & F_0 \ar[r] & X \ar[r] & 0
}
$$
in $\C$. It is thus apparent, what the composition of morphisms, and the identity morphisms should be. Thus we obtain the \textit{category $\mathcal Ext^n_\C(X,Y)$ of $n$-extensions} of $X$ by $Y$. We define $\mathcal Ext^0_\C(X,Y)$ to be the discrete category $\Hom_\C(X,Y)$. Truncation of the object $X$ on the right yields a functor $(-)^\natural$ from $\mathcal Ext^n_\C(X,Y)$ into the category of complexes over $\C$ whose homology is concentrated in degree $0$.

From now on, let us assume that $\C$ is
\begin{enumerate}[\rm(1)]
\item closed under kernels of epimorphisms (that is, $f \in \C$ is an admissible epimorphism if, and only if, $f$ is an epimorphism in $\A_\C$), or
\item closed under cokernels of monomorphisms (that is, $f \in \C$ is an admissible monomorphism if, and only if, $f$ is a monomorphism in $\A_\C$).
\end{enumerate}
Let $S$ be an admissible $n$-extension of $X$ by $Y$. The following construction from \cite{He14b} gives rise to a homomorphism
$$
u_\C : \Ext^{n-1}_\C(X,Y) \longrightarrow \pi_1(\mathcal Ext^n_\C(X,Y), S)
$$
of groups. Let us denote by $S_n$ the \textit{trivial} $n$-extension of $X$ by $Y$. For $n = 1$, this is the split extension $0 \rightarrow Y \rightarrow Y \oplus X \rightarrow X \rightarrow 0$, whereas for $n \neq 1$,
$$
\xymatrix@C=18pt{
S_n & \equiv & 0 \ar[r] & Y \ar@{=}[r] & Y \ar[r] & 0 \ar[r] & \cdots \ar[r] & 0 \ar[r] & X \ar@{=}[r] & X \ar[r] & 0 \, .
}
$$
If $f$ belongs to $\Hom_\C(X,Y)$, and $0 \rightarrow Y \xrightarrow{d} \mathbb E \rightarrow X \rightarrow 0$ to $\mathcal Ext^{n-1}_\C(X,Y)$ for $n \neq 1$, the following diagrams define respective loops in $\mathcal Ext^1_\C(X,Y)$ and $\mathcal Ext^{n}_\C(X,Y)$.
$$
\xymatrix@C=18pt{
0 \ar[r] & Y \ar[r] \ar@{=}[d] & Y \oplus X \ar[d]^{
\left[\begin{smallmatrix}
1 & f\\
0 & 1
\end{smallmatrix}\right]
} \ar[r] & X \ar[r] \ar@{=}[d] & 0\\
0 \ar[r] & Y \ar[r] & Y \oplus X \ar[r] & X \ar[r] & 0
}
$$
$$
\xymatrix{
0 \ar[r] & Y \ar@{=}[r] \ar@{=}[d] & Y \ar[r] & 0 \ar[r] & \cdots \ar[r] & 0 \ar[r] & X \ar@{=}[r] & X \ar[r] \ar@{=}[d] & 0\\
0 \ar[r] & Y \ar[r]^-{\left[\begin{smallmatrix}
1 & 1
\end{smallmatrix}\right]^t} \ar@{=}[d] & Y \oplus Y \ar[u]^-{\left[\begin{smallmatrix}
1 & 0\\
\end{smallmatrix}\right]} \ar[d]_-{\left[\begin{smallmatrix}
0 & 1
\end{smallmatrix}\right]} \ar[r]^-{\left[\begin{smallmatrix}
d & -d
\end{smallmatrix}\right]} & E_{n-2} \ar[d] \ar[u] \ar[r] & \cdots \ar[r] & E_1 \ar[d] \ar[r] \ar[u] & E_0 \ar[r] \ar[d] \ar[u] & X \ar[r] \ar@{=}[d] & 0\\
0 \ar[r] & Y \ar@{=}[r] & Y \ar[r] & 0 \ar[r] & \cdots \ar[r] & 0 \ar[r] & X \ar@{=}[r] & X \ar[r] & 0
}
$$
Thus we obtain a (well-defined) map $u^\circ_\C : \Ext^{n-1}_\C(X,Y) \longrightarrow \pi_1(\mathcal Ext^n_\C(X,Y), S_n)$ for each $n \geqslant 1$. Now, if $S$ is any $n$-extension, taking the Baer sum $- \boxplus S$ with $S$ is an endofunctor on $\mathcal Ext^n_\C(X,Y)$ and thus gives rise to an endofunctor of the corresponding fundamental groupoid. We therefore obtain a group isomorphism
$$
\pi_1 (\mathcal Ext^n_\C(X,Y), S_n) \longrightarrow \pi_1 (\mathcal Ext^n_\C(X,Y), S_n \boxplus S) \, .
$$
The extensions $S$ and $S_n \boxplus S$ are connected by a morphism in $\mathcal Ext^n_\C(X,Y)$ and conjugation with this morphism finally leads to an isomorphism
$$
\pi_1 (\mathcal Ext^n_\C(X,Y), S_n) \longrightarrow \pi_1 (\mathcal Ext^n_\C(X,Y), S) \, .
$$
We let $u_\C$ be the composition of this isomorphism with $u^\circ_\C$. The map $u_\C$ is bijective since $u^\circ_\C$ is (cf.\,\cite[Thm.\,3.1.8 and Prop.\,3.2.5]{He14b}) and hence we obtain the following Theorem (attributed to V.\,Retakh).
\end{nn}

\begin{thm}[see {\cite{He14b} and \cite{Re86}}]\label{thm:retakh}
Let $\C$ be an exact $K$-category as above. Let $n \geqslant 1$ be an integer, $X$, $Y$ be objects in $\C$, and let $S$ be an admissible $n$-extension of $X$ by $Y$. Then there is an isomorphism
$$
\Ext^{n-1}_\C(X,Y) \longrightarrow \pi_1(\mathcal Ext^n_\C(X,Y), S)
$$
of groups. In particular, $\pi_1(\mathcal Ext^{n}_\C(X,Y), S)$ is abelian.
\end{thm}

\begin{nn}
If $\C$ is the category of modules over a ring, the isomorphism $u_\C$ admits a very nice explicit formulation in terms of projective resolutions, as presented in \cite{Sch98}. We will recall it, along with some of its properties. Let $R$ be any ring and let $n \geqslant 0$ be an integer. Fix two $R$-modules $U$ and $V$. We want to describe an isomorphism
$$
\mu: \Ext^n_R(U,V) \cong H^n(\Hom_{R}(\mathbb PU, V)) \longrightarrow \pi_1(\mathcal Ext^{n+1}_{R}(U,V), S)
$$
of groups for every base point $S$ and any fixed projective resolution $\mathbb PU \rightarrow U \rightarrow 0$ of $U$ over $R$. Assume that the projective resolution $\mathbb PU \rightarrow U \rightarrow 0$ of $U$ is given as follows.
$$
\xymatrix@C=20pt{
\cdots \ar[r]^-{\pi_{n+2}} & P_{n+1} \ar[r]^-{\pi_{n+1}} & P_{n} \ar[r]^-{\pi_{n}} & P_{n-1} \ar[r]^-{\pi_{n-1}} & \cdots \ar[r]^-{\pi_1} & P_0 \ar[r]^-{\pi_0} & U \ar[r] & 0
}
$$
Fix an $R$-linear map $\varphi: P_{n+1} \rightarrow V$ satisfying $\varphi \circ \pi_{n+2} = 0$. The pushout diagram
$$
\xymatrix@C=20pt{
\cdots \ar[r]^-{\pi_{n+2}} & P_{n+1} \ar[d]_-{\varphi} \ar[r]^-{\pi_{n+1}} & P_{n} \ar[r]^-{\pi_{n}} \ar[d] & P_{n-1} \ar[r]^-{\pi_{n-1}} \ar@{=}[d] & \cdots \ar[r]^-{\pi_1} & P_0 \ar[r]^-{\pi_0} \ar@{=}[d] & U \ar[r] \ar@{=}[d] & 0\\
0 \ar[r] & V \ar[r]^-{p_{n+1}} & P \ar[r]^-{p_n} & P_{n-1} \ar[r]^-{p_{n-1}} & \cdots \ar[r]^-{p_1} & P_0 \ar[r]^-{p_0} & U  \ar[r] & 0
}
$$
has exact rows. We may regard $P$ as the quotient
$$
P = \frac{V \oplus P_{n}}{\{(\varphi(x), -\pi_{n+1}(x)) \mid x \in P_{n+1}\}} = \Coker(\varphi \oplus (-\pi_{n+1}))
$$
and hence express the maps $p_{n+1}: Y \rightarrow P$, $p_n: P \rightarrow P_{n-1}$ as $p_{n+1}(v) = (v,0)$ and $p_n(v,p) = \pi_{n}(p)$ (for $p \in P_n$, $v \in V$).

If $\psi: P_{n} \rightarrow V$ is an $R$-module homomorphism, then $\varphi' = \varphi + \psi \circ \pi_{n+1}$ will also satisfy $\varphi' \circ \pi_{n+2} = 0$. Hence we may consider $S(\varphi')$ as well. The assignment $(v,p) \mapsto (v - \psi(p),p)$ gives rise to a well-defined map
$$
\frac{V \oplus P_{n}}{\{(\varphi(x), -\pi_{n+1}(x)) \mid x \in P_{n+1}\}} \longrightarrow \frac{V \oplus P_{n} }{\{(\varphi'(x), -\pi_{n+1}(x)) \mid x \in P_{n+1}\}} \ ,
$$
which itself defines a morphism $\mu(\psi): S(\varphi) \rightarrow S(\varphi + \psi \circ \pi_{n+1})$ of $(n+1)$-extensions. If we chose $\psi$ such that $\psi \circ \pi_{n+1} = 0$, $\mu(\psi)$ will be an endomorphism of $S(\varphi)$ in the category $\mathcal Ext^{n+1}_{R}(U,V)$, i.e., a loop of length $1$ based at $S(\varphi)$. In \cite[Sec.\,4]{Sch98} it is shown that hereby one obtains a well-defined map
$$
\mu: H^n(\Hom_{R}(\mathbb PU, V)) \longrightarrow \pi_1(\mathcal Ext^{n+1}_{R}(U,V), S(\varphi))
$$
which is an isomorphism (cf. \cite[Thm.\,1.1]{Sch98}). Given any $(n+ 1)$-extension $S$ of $U$ by $V$, say
$$
S \quad \equiv \quad 0 \longrightarrow V \longrightarrow E_{n} \longrightarrow \cdots \longrightarrow E_0 \longrightarrow U \longrightarrow 0 \ ,
$$
there is a map $\Phi: \mathbb PU \rightarrow S^\natural$ of complexes lifting $\id_U$, where $S^\natural$ denotes the truncated complex 
$$
0 \longrightarrow V \longrightarrow E_{n} \longrightarrow \cdots \longrightarrow E_0 \ .
$$
The component map $\varphi_{n+1}: P_{n+1} \rightarrow V$ is a cocycle in $\Hom_R(\mathbb PU, V)$. The universal property of the pushout
$$
V \oplus_{P_{n+1}} P_{n} = \frac{V \oplus P_{n}}{\{(\varphi_{n+1}(x), -\pi_{n+1}(x)) \mid x \in P_{n+1}\}} = \Coker(\varphi_{n+1} \oplus (-\pi_{n+1}))
$$
yields a map $\tilde{\varphi}_{n}: V \oplus_{P_{n+1}} P_{n} \rightarrow E_n$ such that
$$
\xymatrix@C=20pt{
0 \ar[r] & V \ar[r]^-{p_{n+1}} \ar@{=}[d] & V \oplus_{P_{n+1}} P_{n} \ar[r]^-{p_n} \ar[d]_-{\tilde{\varphi}_{n}} & P_{n-1} \ar[r]^-{p_{n-1}} \ar[d]_-{\varphi_{n-1}} & \cdots \ar[r]^-{p_1} & P_0 \ar[r]^-{p_0} \ar[d]^-{\varphi_0} & U \ar[r] \ar@{=}[d] & 0\\
0 \ar[r] & V \ar[r]^-{e_{n+1}} & E_n \ar[r]^-{e_n} & E_{n-1} \ar[r]^-{e_{n-1}} & \cdots \ar[r]^-{e_1} & E_0 \ar[r]^-{e_0} & U  \ar[r] & 0
}
$$
commutes, and hence a map $\tilde{\Phi}: S(\varphi_{n+1}) \rightarrow S$ of extensions. Observe that the map $\tilde{\varphi}_{n}$ is induced by the map
$$
V \oplus P_n \longrightarrow E_n, \ (v,p) \mapsto e_{n+1}(v) + \varphi_n(p) \ .
$$
Conjugation with (the equivalence class of) $\tilde{\Phi}$ delivers an isomorphism
$$
c_\Phi: \pi_1(\mathcal Ext^{n+1}_{R}(U,V), S(\varphi_{n+1})) \longrightarrow \pi_1(\mathcal Ext^{n+1}_{R}(U,V), S)
$$
which depends on the chosen lifting $\Phi$. However, the composition $c_\Phi \circ \mu_{\varphi = \varphi_{n+1}}$ does not (use the remarks before \cite[Lem.\,4.3]{Sch98}). Hence we may (and will) regard $\mu$ as an isomorphism
$$
\mu: H^n(\Hom_{R}(\mathbb PU, V)) \longrightarrow \pi_1(\mathcal Ext^{n+1}_{R}(U,V), S) \ .
$$

\begin{lem}[{see \cite[Lem.\,5.3.3]{He14b}}]
Let $\C$ be the category $\Mod(R)$. The maps $\mu$ and $u_\C$ are equal $($after identifying $\Ext^n_R(U,V)$ with $H^n(\Hom_{R}(\mathbb PU, V)))$.
\end{lem}

For a loop $w$ of length $2$ in $\mathcal Ext^{n+1}_R(U,V)$ we explicitly describe a preimage of its equivalence class in $\pi_1(\mathcal Ext^{n+1}_R(U,V), S)$ under the isomorphism $\mu$. Fix an $(n+1)$-extension $S$ of $U$ by $V$, and a map $\Phi: \mathbb PU \rightarrow S^\natural$ of complexes lifting $\id_U$. Assume that $f: S \rightarrow T$ and $g: S \rightarrow T$ are morphisms in $\mathcal Ext^{n+1}_R(U,V)$. They define a loop $S \rightarrow T \leftarrow S$ in $\mathcal Ext^{n+1}_R(U,V)$ based at $S$. The compositions $f \circ \Phi$ and $g \circ \Phi$ define morphisms $\mathbb PU \rightarrow T^\natural$ of complexes lifting $\id_U$. Hence the difference $f \circ \Phi - g \circ \Phi$ is null-homotopic, say via the null-homotopy $s_i: P_i \rightarrow F_{i+1}$ ($i \geqslant 0$). Note that $s_n: P_n \rightarrow V$ has to be a cocycle in $\Hom_R(\mathbb PU, V)$.

\begin{lem}[{see \cite[Lem.\,4.3]{Sch98}}]\label{lem:preimage_mu}
Keep the notations from above. The equivalence class of the cocycle $s_{n}: P_n \rightarrow V$ is mapped, by $\mu$, to the equivalence class of the loop
$$
\xymatrix@C=20pt{S \ar[r]^-f & T & \ar[l]_-g S}
$$
in $\pi_1(\mathcal Ext^{n+1}_R(U,V), S)$.
\end{lem}
\end{nn}

%%%%%%%%%%%%%%%%%%%%%%%%%%%%%%%%
%%% The $M$-relative center
%%%%%%%%%%%%%%%%%%%%%%%%%%%%%%%%

\section{The $M$-relative center}\label{sec:Mrel}

Fix a unital and associative $K$-algebra $A$. The symbol $\otimes$ will always stand for $\otimes_K$, i.e., the tensor product over the base ring $K$. Let $A^\ev = A \otimes A^\op$ be the \textit{enveloping algebra} of $A$ over $K$, with factorwise multiplication. It is a very well-known fact, that $A$-bimodules with symmetric $K$-action bijectively correspond to left modules over $A^\ev$. In this section, we will recall the definition of Hochschild cohomology, and its (higher) structures, as they were introduced in \cite{CaEi56}, \cite{Ge63} and \cite{Ho45}.

\subsection{Reminder on Hochschild cohomology}
Let $M$ be a $A^\ev$-module. The \textit{Hochschild cocomplex} $\mathbb C(A,M) = (C^\ast(A,M), \partial_M)$ with coefficients in $M$ is the cocomplex concentrated in non-negative degrees which is given by
$$
C^n(A,M) = \Hom_K(A^{\otimes n},M) \quad \text{(for $n \geqslant 0$)}
$$
and $\partial_M^n : C^n(A,M) \rightarrow C^{n+1}(A,M)$,
\begin{align*}
\partial_M^n(f)(a_1 \otimes \cdots a_{n+1}) = & a_1 f(a_2 \otimes \cdots \otimes a_{n+1})\\ &+ \sum_{i=1}^{n}(-1)^n f(a_1 \otimes \cdots \otimes a_{i-1} \otimes a_i a_{i+1} \otimes a_{i+1} \otimes \cdots \otimes a_{n+1})\\
&+ (-1)^{n+1}f(a_1 \otimes \cdots \otimes a_{n})a_{n+1}.
\end{align*}
The Hochschild cocomplex admits an exterior pairing, in that there is a map
$$
\smallsmile \,: C^m(A,M) \times C^n(A,N) \longrightarrow C^{m+n}(A,M \otimes_A N)
$$
for any pair $M, N \in \Mod(A^\ev)$. Namely,
$$
(f \smallsmile g)(a_1 \otimes \cdots a_{m + n}) = f(a_1 \otimes \cdots \otimes a_{m}) \otimes g(a_{m+1} \otimes \cdots \otimes a_{m+n}).
$$
In fact, through $\smallsmile$, and after identifying $A \otimes_A A \cong A$, $M \otimes_A A \cong M \cong A \otimes_A M$, the DG module $\mathbb C(A,A)$ is a DG $K$-algebra, and $\mathbb C(A,M)$ will be a left and a right DG module over it. In particular, the cohomology of $\mathbb C(A,A)$ will be a graded $K$-algebra.

\begin{defn}
For an integer $n \geqslant 0$, the \textit{$n$-th Hochschild cohomology module} with coefficients in $M$ is given by
$$
\HH^n(A,M) = H^n \mathbb C(A,M).
$$
The graded module
$$
\HH^\ast(A,M) = \bigoplus_{n \geqslant 0}\HH^n(A,M)
$$
is the \textit{Hochschild cohomology module} of $A$ with coefficients in $M$. We abbriviate $\HH^n(A) = \HH^n(A,A)$ for $n \geqslant 0$ and call $\HH^\ast(A)$ the \textit{Hochschild cohomology algebra} of $A$.
\end{defn}

\begin{nn}
There are two immediate observations.
\begin{enumerate}[\rm(1)]
\item The module $\HH^0(A,M)$ coincides with $M^A$, where
$$
M^A = \{m \in M \mid am = ma \text{ for all $a \in A$}\}.
$$
In particular, $\HH^0(A) = Z(A)$ is the center of $A$.
\item The module $\HH^1(A,M)$ coincides with the module $\Out_K(A,M)$ of \textit{outer derivations} which is given by
$$
\Out_K(A,M) = \frac{\Der_K(A,M)}{\Inn_K(A,M)},
$$
where $\Der_K(A,M)$ denotes the derivations of $M$, and $\Inn_K(A,M)$ the submodule of inner derivations.
\end{enumerate}
\end{nn}

\begin{nn}
The \textit{bar resolution} $\mathbb BA = (B_\ast, \beta_\ast)$ of $A$, is the following exact resolution of $A$ by $A^\ev$-modules. To begin with, $B_n = A^{\otimes(n+2)}$ is the $(n+2)$-fold tensor product of $A$ with itself (over $K$ and for $n \geqslant 0$). The $A^\ev$-linear map $\beta_{n+1}: B_{n+1} \rightarrow B_n$,
$$
\beta_{n+1}(a_0 \otimes \cdots \otimes a_{n+2}) = \sum_{i=0}^{n+1}(-1)^i a_0 \otimes \cdots \otimes a_i a_{i+1} \otimes \cdots \otimes a_{n+2},
$$
turns $\mathbb BA = (B_\ast, \beta_\ast)$ into a complex, which is acyclic in all degrees but in degree $0$, wherein its homology is isomorphic to $A$. The multiplication map $\mu: A \otimes A \rightarrow A$ provides a suitable augmentation $\mathbb BA \rightarrow A \rightarrow 0$.

Now, the adjunction isomorphism
$$
\Hom_{A^\ev}(A^{\otimes(n+2)}, M) \longrightarrow \Hom_K(A^{\otimes n}, M), \ \varphi \mapsto \varphi(1 \otimes - \otimes \cdots \otimes - \otimes 1)
$$
is compatible with the differentials $\partial_M$ and $\Hom_{A^\ev}(\mathbb \beta_\ast, M)$, that is, the cocomplexes $\mathbb C(A,M)$ and $\Hom_{A^\ev}(\mathbb BA, M)$ are isomorphic. Thus $\HH^\ast(A,M) \cong H^\ast \Hom_{A^\ev}(\mathbb BA, M)$ and, since $\mathbb BA \rightarrow A \rightarrow 0$ is exact, there is a graded map
$$
\chi_{M}^\ast : \HH^\ast(A,M) \longrightarrow \Ext^\ast_{A^\ev}(A,M).
$$
It is given by sending a cocycle $\varphi \in \Ker \Hom_{A^\ev}(\beta_{n+1}, M)$ to the equivalence class defined by the lower sequence in the pushout diagram below.
$$
\xymatrix@C=20pt{
\cdots \ar[r]^-{\beta_{n+1}} & A^{\otimes (n + 2)} \ar[r]^-{\beta_{n}} \ar[d]_-\varphi & A^{\otimes (n + 1)} \ar[r]^-{\beta_{n-1}} \ar[d] & \cdots \ar[r]^-{\beta_2} & A^{\otimes 3} \ar[r]^-{\beta_1} \ar@{=}[d] & A \otimes A \ar[r]^-{\beta_0} \ar@{=}[d] & A \ar[r] \ar@{=}[d] & 0\\
0 \ar[r] & M \ar[r] & Q \ar[r] & \cdots \ar[r]^-{\beta_2} & A^{\otimes 3} \ar[r]^-{\beta_1} & A \otimes A \ar[r]^-{\beta_0} & A \ar[r] & 0
}
$$
This map respects the graded ring structures if $M = A$, but will in general not be a bijection (as we will remark later, $\HH^\ast(A)$ is always \textit{graded commutative}, whereas a sufficient criterion for $\Ext^\ast_{A^\ev}(A,A)$ being graded commutative is $\Tor_i^K(A,A) = 0$ for all $i > 0$; see \cite[Sec.\,2.2]{BuFl08} and \cite{SnSo04}). However, one easily checks that $\mathbb BA$ will be a resolution by projective $A^\ev$-modules, if $A$ is projective over $K$. Thus the above map is going to be an ismorphism,
$$
\chi^\ast_M : \HH^\ast(A,M) \xrightarrow{\ \sim \ } \Ext^\ast_{A^\ev}(A,M),
$$
if $A$ is $K$-projective.
\end{nn}

\subsection{The Gerstenhaber bracket in Hochschild cohomology} Let us first recall the definition of a Gerstenhaber algebra over $K$.

\begin{defn}\label{def:galgebra}
Let $G = \bigoplus_{n \in \Z}{G^n}$ be a graded $K$-algebra. Further, let $\{-,-\}: G \times G \rightarrow G$ be a $K$-bilinear map of degree $-1$ (that is, $[a,b] \in G^{\abs{a}+\abs{b}-1}$ for all homogeneous $a,b \in G$). The pair $(G,\{-,-\})$ is a \textit{Gerstenhaber algebra} over $K$ if
\begin{enumerate}[\rm(G1)]
\item $G$ is graded commutative, i.e., $ab = (-1)^{\abs{a}\abs{b}}ba$ for all homogeneous $a,b \in G$;
\item $\{a,b\} = -(-1)^{(\abs{a}-1)(\abs{b}-1)}\{b,a\}$ for all homogeneous $a,b \in G$;
\item $\{a,a\} = 0$ for all homogeneous $a \in G$ of odd degree;
\item $\{\{a,a\},a\} = 0$ for all homogeneous $a \in G$ of even degree;
\item\label{def:galgebra:5} the graded Jacobi identity holds:
$$
\{a,\{b,c\}\} = \{\{a,b\},c\} + (-1)^{(\abs{a}-1)(\abs{b}-1)}\{b,\{a,c\}\}
$$
for all homogeneous $a,b,c \in G$;
\item\label{def:galgebra:6} the graded Poisson identity holds:
$$
\{a,bc\} = \{a,b\}c + (-1)^{(\abs{a}-1)\abs{b}} b\{a,c\}
$$
for all homogeneous $a,b,c \in G$.
\end{enumerate}
\end{defn}

\begin{nn}
The map $\{-,-\}$ is called a \textit{Gerstenhaber bracket} for $G$. Note that any graded commutative $K$-algebra can be viewed as a Gerstenhaber algebra over $K$ with trivial bracket.

The axioms (G\ref{def:galgebra:5}) and (G\ref{def:galgebra:6}) may be read as follows: The graded Jacobi identity measures the failure of $\{-,-\}$ from being associative, whereas the graded Poisson identity translates to $\{a,-\}$ being a graded derivation of $G$ of degree $\abs{a} - 1$ (for $a \in G$ homogeneous).
\end{nn}

\begin{nn}\label{nn:Gbracket}
Let $M$ be a fixed $A^\ev$-module and $f \in C^m(A,M)$, $g \in C^n(A,A)$ for integers $m, n \geqslant 0$. For $i=1, \dots, m$, let $f \bullet_i g \in C^{m + n -1}(A,M)$
\begin{align*}
&\qquad(f \bullet_i g)(a_1 \otimes \cdots \otimes a_{m+n-1})\\ = f(a_1 \otimes \cdots & \otimes a_{i-1} \otimes g(a_i \otimes \cdots \otimes a_{i+n-1}) \otimes a_{i+n} \otimes \cdots \otimes a_{m+n-1}).
\end{align*}
Denote by $f \bullet g \in C^{m+n-1}(A,M)$ the alternating sum of the $f \bullet_i g$:
\begin{align*}
f \bullet g = \sum_{i=1}^m (-1)^{(i-1)(n-1)}f \bullet_i g.
\end{align*}
The product $\bullet$ is, in general, non-unital and highly non-associative. However, the external pairing $\smallsmile$ and $\bullet$ are related by the following fundamental formula:
\begin{equation}\tag{$\dagger$}\label{eq:fundform}
\partial_M(f \bullet g) + (-1)^n \partial_M(f) \bullet g = f \bullet \partial_A(g) + (-1)^{n}[g \smallsmile f - (-1)^{mn}f \smallsmile g].
\end{equation}
See \cite[Thm.\,3]{Ge63} for a proof.
\end{nn}

\begin{nn}
The fundamental formula (\ref{eq:fundform}) yields two important insights.
\begin{enumerate}[\rm(1)]
\item If $M = A$, and $f$ and $g$ are \textit{cocycles} (i.e., $\partial_A(f) = 0 = \partial_A(g)$), then $f \smallsmile g = (-1)^{mn} g \smallsmile f$. It follows that $\HH^\ast(A)$ is a graded commutative $K$-algebra.
\item The map
\begin{align*}
\{-,-&\} : C^m(A,A) \times C^n(A,A) \longrightarrow C^{m+n-1}(A,A)\\
&\{f,g\} = f \bullet g - (-1)^{(m-1)(n-1)}g \bullet f
\end{align*}
induces a well-defined map
$$
\{-,-\}: \HH^m(A) \times \HH^n(A) \longrightarrow \HH^{m+n-1}(A).
$$
\end{enumerate}
It is the main observation of \cite{Ge63} that the hereby obtained triple $(\HH^\ast(A), \smallsmile, \{-,-\})$ is a Gerstenhaber algebra over $K$, in the sense of Definition \ref{def:galgebra}.
\end{nn}

\subsection{The action of the $M$-relative center} For this entire subsection, we fix a $K$-algebra $A$ and an $A^\ev$-module $M$. We make the following crucial definition, inspired by \cite{Si67} and \cite[Chap.\,XI]{St75}.

\begin{defn}
The \textit{$M$-relative center} of $A$ is given by
$$
Z_M(A) = \{ a \in Z(A) \mid am = ma \text{ for all $m \in M$} \}.
$$
It is a sub-$K$-algebra of $Z(A)$, and equality holds, if $M = A$.
\end{defn}

\begin{exa}
If the $K$-algebra $A$ is commutative, every $A$-module $M$ may be regarded as an $A^\ev$-module. As such, $Z_M(A) = A$. In fact, an $A^\ev$-module $N$ arises from an $A$-module in such a way if, and only if, $Z_N(A) = A$.
\end{exa}

We get the following obvious criterion as to when the relative center agrees with the whole center of the underlying algebra.

\begin{lem}\label{lem:MrelcentHH}
Let $M$ be an $A^\ev$-module. Then $Z(A) = Z_M(A)$ holds true if, and only if, $\HH^0(Z(A),M) = M^{Z(A)} = M$, which is if, and only if, $\Inn_K(Z(A), M) = 0$.\qed
\end{lem}

\begin{lem}\label{lem:Mrelcent}
The following statements are equivalent.
\begin{enumerate}[\rm(1)]
\item\label{lem:Mrelcent:1} $Z(A) = Z_{A \otimes A}(A)$.
\item\label{lem:Mrelcent:2} $Z(A) = Z_M(A)$ for all $M \in \Mod(A^\ev)$.
\item\label{lem:Mrelcent:3} $(\Mod(Z(A)^\ev), \otimes_{Z(A)}, Z(A))$ is braided monoidal.
\item\label{lem:Mrelcent:4} The multiplication map $Z(A) \otimes Z(A) \rightarrow Z(A)$ is an isomorphism.
\item\label{lem:Mrelcent:5} The unit $K \rightarrow Z(A)$ is an epimorphism in the category of rings.
\end{enumerate}
\end{lem}

\begin{proof}
The implication (\ref{lem:Mrelcent:2})$\, \Longrightarrow \,$(\ref{lem:Mrelcent:1}) is trivial. $Z(A) = Z_{A \otimes A}(A)$ means that $a \otimes 1 = 1 \otimes a$ for all $a \in Z(A)$. Therefore, if $M$ is any $A^\ev$-module, we have $am = (a \otimes 1)m = (1 \otimes a)m = ma$ for all $a \in Z(A)$ and all $m \in M$. Thus the first two items are equivalent.

If $a \otimes 1 = 1 \otimes a$ for all $a \in A$, then $\mathbf r = 1 \otimes 1 \otimes 1 \in Z(A) \otimes Z(A) \otimes Z(A)$ evidently satisfies the equations mentioned in Example \ref{exa:monoidal}(\ref{exa:monoidal:1}), that is, $(\Mod(Z(A)^\ev), \otimes_{Z(A)}, Z(A))$ is braided monoidal. Conversly, if the monoidal category $(\Mod(Z(A)^\ev), \otimes_{Z(A)}, Z(A))$ is braided monoidal, the element $\mathbf r = 1 \otimes 1 \otimes 1$ solves the equations in Example \ref{exa:monoidal}(\ref{exa:monoidal:1}) for the algebra $Z(A)$ (see \cite[Prop.\,3.3]{AgCaMi12}). From the first of these equations, it immediately follows that $a \otimes 1 = 1 \otimes a$, hence $(\ref{lem:Mrelcent:1}) \, \Longleftrightarrow \, (\ref{lem:Mrelcent:3})$.

Lastly, $a \otimes 1 = 1 \otimes a$ for all $a \in Z(A)$ is equivalent to the restricted multiplication map $\mu \colon Z(A)\otimes Z(A) \rightarrow Z(A)$ being an isomorphism, as $\Ker(\mu)$ is generated by $a \otimes 1 - 1 \otimes a$; cf.\,Lemma \ref{lem:omegagen}. On the other hand, $\mu : Z(A) \otimes Z(A) \rightarrow Z(A)$ is an isomorphism if, and only if, the unit map $K \rightarrow Z(A)$ is an epimorphism in the category of rings; see \cite[Prop.\,1.1]{Si67}. Thus, $(\ref{lem:Mrelcent:1}) \, \Longleftrightarrow \, (\ref{lem:Mrelcent:4}) \, \Longleftrightarrow \, (\ref{lem:Mrelcent:5})$.
\end{proof}

\begin{nn}
In order to describe the desired action, we need a map
$$
[-,-]_M : \HH^n(A,M) \times Z_M(A) \longrightarrow \HH^{n-1}(A,M) \quad \text{(for $n \geqslant 0$)}.
$$
Let $f \in \Hom_K(A^{\otimes n}, M)$ be a $K$-linear homomorphism. For $z \in Z_M(A)$ and $1 \leqslant i \leqslant n$, we obtain a $K$-linear homomorphism $A^{\otimes(n-1)} \rightarrow M$ by putting
$$
(f \bullet_i z)(a_1 \otimes \cdots \otimes a_n) = f(a_1 \otimes \cdots a_{i-1} \otimes z \otimes a_{i+1} \otimes \cdots \otimes a_n) 
$$
for $a_1, \dots, a_n \in A$, thus simply copying the formula in \ref{nn:Gbracket}. By taking the alternating sum,
$$
f \bullet z = \sum_{i=1}^n{(-1)^{(i-1)}f \bullet_i z},
$$
we arrive at a map $A^{\otimes(n-1)} \rightarrow M$ only depending on $f$ and $z$. Although he never stated it in this way, we attribute the following result to Gerstenhaber, since it entirely bases on the ideas and techniques which can be found in \cite{Ge63}. An idea of the proof will be provided for the convenience of the reader.

Recall beforehand, that a \textit{right Gerstenhaber module} over a Gerstenhaber $K$-algebra $(G, \{-,-\})$ is a graded $G$-$G$-bimodule $U$ along with a $K$-bilinear map
$$
\langle -,-\rangle : U^m \times G^n \longrightarrow U^{m + n - 1} \quad \text{(for $m, n \in \mathbb Z$)},
$$
such that for $u \in U$, $a, b \in G$ homogeneous, $a u = (-1)^{\abs{a}\abs{u}}ua$,
$$
\langle u ,\{a,b\}\rangle = \langle \langle u, a \rangle, b \rangle - (-1)^{(\abs{a}-1)(\abs{b}-1)} \langle \langle u,b\rangle,a\rangle \quad
$$
and
$$
\langle u , ab\rangle = \langle u, a\rangle b + (-1)^{(\abs{u}-1)\abs{a}}a \langle u, b\rangle .
$$
The latter equality means, that $\langle u, -\rangle : G \rightarrow U$ is a graded derivation of degree $\abs{u}-1$.
\end{nn}

\begin{thm}
With notation as above, the assignment
$$
(f,z) \mapsto f \bullet z = \sum_{i=1}^n{(-1)^{(i-1)}f \bullet_i z}
$$
induces a well-defined $K$-bilinear map
$$
[-,-] = [-,-]_M : \HH^n(A,M) \times Z_M(A) \longrightarrow \HH^{n-1}(A,M)
$$
turning $\HH^\ast(A,M)$ into a right Gerstenhaber module over the Gerstenhaber algebra $Z_M(A)$ $($concentrated in degree $0$ with trivial bracket$)$, that is,
$$
0 = [[\alpha, z], z'] + [[\alpha, z'],z]
$$
and
$$
[\alpha, zz'] = [\alpha, z] z' + z[\alpha, z']
$$
for $\alpha \in \HH^\ast(A,M)$ homogeneous and $z,z' \in Z_M(A))$.
\end{thm}

\begin{proof}
Let $z \in Z_M(A)$. By the fundamental formula (\ref{eq:fundform}), we have
$$
\partial_M(f \bullet z) = f \smallsmile z - z \smallsmile f = 0
$$
for every cocycle $f \in C^n(A,M)$, that is, we obtain a well-defined map as claimed. It is evident that the external pairing $\smallsmile$ turns $\HH^\ast(A,M)$ into a $Z_M(A)$-$Z_M(A)$-bimodule with $z \alpha = \alpha z$ for all $\alpha \in \HH^\ast(A,M)$ homogeneous. The formulas $0 = [[\alpha, z], z'] + [[\alpha, z'],z]$ and $[\alpha, zz'] = [\alpha, z] z' + (-1)^{\abs \alpha} z[\alpha, z']$ follow from \cite[Thm.\,2(ii)]{Ge63} and \cite[Thm.\,5]{Ge63}.
\end{proof}

\begin{cor}\label{cor:bialg}
Let $\Gamma$ be a bialgebra over $K$ which is is projective as a $K$-module. Let $M$ a $\Gamma$-module, and let $\OH^\ast(\Gamma, M) = \Ext^\ast_\Gamma(K,M)$ be the corresponding graded cohomology module with coefficients in $M$. Then $\OH^\ast(\Gamma,M)$ is a right Gerstenhaber module over the Gerstenhaber algebra $Z_{M \otimes \Gamma}(\Gamma)$. In particular, the cohomology ring $\OH^\ast(\Gamma, K)$ of $\Gamma$ is a right Gerstenhaber module over $Z(\Gamma)$.
\end{cor}

\begin{proof}
In the following, we will view $M \otimes N$ for $M, N \in \Mod(\Gamma)$ as a left $\Gamma$-module through $\Delta: \Gamma \rightarrow \Gamma \otimes \Gamma$. The cohomology module $\OH^\ast(\Gamma, M)$ embeds into $\HH^\ast(\Gamma, M \otimes \Gamma)$ thanks to the map $r : \Ext^n_\Gamma(K,M) \rightarrow \Ext^n_{\Gamma^\ev}(\Gamma, M \otimes \Gamma)$ taking an $n$-extension $S$ of $K$ by $M$ to the bimodule $n$-extension $S \otimes \Gamma$ of $\Gamma$ by $M \otimes \Gamma$. The map $r$ splits, and a left inverse is given by the following map $s: \Ext^n_{\Gamma^\ev}(\Gamma, \Gamma \otimes M) \rightarrow \Ext^n_\Gamma(K,M)$. Assume that an element $\alpha'$ in $\Ext^n_{\Gamma^\ev}(\Gamma, M \otimes \Gamma)$ is represented by a cocycle $\varphi : \Gamma^{\otimes (n+2)} \rightarrow M \otimes \Gamma$ of $\Hom_{\Gamma^\ev}(\mathbb B\Gamma, M \otimes \Gamma)$. By considering the obvious pushout diagram involving $\varphi$ and the bar resolution's $n$-th differential $\beta_n$, we obtain an $n$-extension
$$
S_\varphi \quad \equiv \quad 0 \longrightarrow M \otimes \Gamma \longrightarrow Q \longrightarrow \Gamma^{\otimes n} \longrightarrow \cdots \longrightarrow \Gamma \otimes \Gamma \longrightarrow \Gamma \longrightarrow 0\,.
$$
As a sequence of right $\Gamma$-modules, this sequence splits, in that each degree is the direct sum of the kernel and the image of the bordering homomorphisms. This follows from the fact that, since $\Gamma$ is $K$-projective, $\Gamma^{\otimes m}$ is $\Gamma$-projective for all $m \geqslant 1$. Thus $S_\varphi \otimes_\Gamma K$ is an $n$-extension of $K$ by $M$, and we let $s(\alpha')$ be its equivalence class in $\Ext^n_\Gamma(K,M)$. 

After all, it now easily follows that $\OH^\ast(\Gamma, M)$ is a right Gerstenhaber module over $Z_{M \otimes \Gamma}(\Gamma)$ through $\langle \alpha, z \rangle = s ([ r(\alpha), z ])$, $\alpha \in \OH^\ast(\Gamma,M)$, $z \in Z_{M \otimes \Gamma}(\Gamma)$, since $s$ satisfies $s(\alpha' \alpha '') = s(\alpha')s(\alpha'')$ for all $\alpha' \in \HH^\ast(\Gamma, M \otimes \Gamma)$, $\alpha'' \in \HH^\ast(\Gamma)$.
\end{proof}

\begin{nn}
Observe that $[-,-]$ will be
$$
\{-,-\}: \HH^n(A) \times \HH^0(A) \longrightarrow \HH^{n-1}(A) \quad \text{(for $n \geqslant 0$)}
$$
if $M = A$, that is, the Gerstenhaber bracket on $\HH^\ast(A)$ in degrees $(n, 0)$ described in the previous subsection. The bracket $[-,-]$ is of fundamental importance, especially in the mentioned case $M = A$.
\end{nn}

\begin{thm}\label{lem:isoHHandOH}
Let $\Gamma$ be a Hopf algebra over $K$. If $\Gamma$ is commutative, quasi-triangular $($that is, the triple $(\Mod(\Gamma), \otimes, K)$ is braided monoidal$)$ and finitely generated projective as a $K$-module, then
\begin{equation}\tag{$\triangle$}\label{eq:lem:isoHHandOH}
\HH^\ast(\Gamma) \cong \OH^\ast(\Gamma, K) \otimes \Gamma
\end{equation}
as graded algebras. Under this isomorphism, the Gerstenhaber bracket $\{-,-\}$ on $\HH^\ast(\Gamma)$ takes the form
$$
\{\alpha \otimes y, \beta \otimes z\} = \{ \alpha \otimes 1, 1 \otimes z\} \cdot (\beta \otimes y) + (-1)^{\abs{\beta}}(\alpha \otimes z) \cdot \{ \beta \otimes 1, 1 \otimes y\} \, .
$$
In particular, the bracket $\{-,-\}$ on $\HH^\ast(\Gamma)$ is completely determined by the map $[-,-] : \HH^\ast(\Gamma) \times \Gamma \rightarrow \HH^{\ast-1}(\Gamma)$.
\end{thm}

\begin{proof}
The ismorphism (\ref{eq:lem:isoHHandOH}) of graded $K$-algebras has been obtained in \cite{Li00}. That the bracket $\{-,-\}$ on $\HH^\ast(\Gamma)$ takes desired form under it, has been varified in \cite[Cor.\,6.4.7]{He14b}. 
\end{proof}

\begin{rem}
If $G$ is a finite abelian group, and if $K$ is a field, the Gerstenhaber bracket on $\HH^\ast(KG)$ has been computed in the recent article \cite{LeZh13}. The computations yield a very nice illustration of the above theorem.
\end{rem}

\begin{nn}
The map $[-,-]$ for $M = A$ also plays a role in Poisson geometry. Following \cite{Xu94}, a \textit{Poisson structure} for an algebra $A$ over a field $K$ of characteristic different from $2$ is an element $\Pi \in \HH^2(A)$ with $\{\Pi,\Pi\} = 0$. As a simple application of Poisson structures, we provide the following result, without proof, and refer to \cite{Xu94} for further reading.
\end{nn}

\begin{thm}
Let $K$ be a field with $\mathrm{char}(K) \neq 2$. Let $A$ be a $K$-algebra and $\Pi \in \HH^2(A)$ be a Poisson structure for $A$. Then the map
$$
\{-,-\}_\Pi : Z(A) \times Z(A) \longrightarrow Z(A),
$$
defined by $\{z,z'\}_\Pi = [[\Pi, z],z'] = \{\{\Pi, z\},z'\}$, is a Poisson bracket for $Z(A)$.
\end{thm}

\begin{nn}
By the theorem of Hochschild-Kostant-Rosenberg, see \cite{HKR62}, we have
\begin{equation}\tag{$\#$}\label{eq:HKR}
\HH^\ast(A) \cong \Lambda_A^\ast \Der_K(A),
\end{equation}
whenever $A$ is smooth over a field $K$ of characteristic zero. The right hand side carries the structure of a Gerstenhaber algebra via the Schouten-Nijenhuis bracket, see \cite{Gi05}, and the isomorphism above respects this additional structure. It is classical, that for a smooth manifold $M$ over $K = \mathbb R$, the algebra $A = C^\infty(M,\mathbb R)$ of smooth functions is a smooth algebra over $\mathbb R$. As a very well-known fact, Poisson brackets for $A$ are in bijective correspondence with those elements of $\Lambda^2_{A} \Der_{\mathbb R}(A)$ whose Schouten-Nijenhuis bracket against themselves vanishes; that is, Poisson brackets for $A$ bijectively correspond to Poisson structures $\Pi \in \HH^2(A)$ by (\ref{eq:HKR}). See \cite{Gi05} and \cite{Ro09} for clear discussions of the topic.
\end{nn}

%%%%%%%%%%%%%%%%%%%%%%%%%%%%%%%%
%%% The action of the $X$-relative center of a monoidal category
%%%%%%%%%%%%%%%%%%%%%%%%%%%%%%%%

\section{The action of the $X$-relative center of a monoidal category}\label{sec:actmonoidal}

\begin{nn}
In degrees $m, n \geqslant 1$, S.\,Schwede offered a categorical interpretation of the Gerstenhaber bracket in Hochschild cohomology; see \cite{Sch98}. We generalized this description to exact monoidal categories in \cite{He14b}. In this section we will include the missing interpretation of the bracket in degrees $(n,0)$ for $n \geqslant 0$. Let $\C$ be a $K$-linear exact tensor category with tensor functor $\otimes = \otimes_\C$ and tensor unit $\bbm1 = \bbm1_\C$. We further assume that $\C$, as an exact category, is closed under kernels of epimorphisms. The basic idea is to construct loops $\Omega_\C(S,f)$ based at $S \# f$ for each choice of an extension $S \in \mathcal Ext^n_\C(\bbm1, \bbm1)$ and a morphism $f \in \Hom_\C(\bbm1, \bbm1)$ which behave well with respect to morphisms of extensions (i.e., if $S \rightarrow S'$ is a morphism in $\mathcal Ext_\C^n(\bbm1,\bbm1)$, then $\Omega_\C(S',f) = c_{f}(\Omega_\C(S,f))$, where $c_{f}$ denotes the conjugation with the morphism $S \circ f \rightarrow S' \circ f$ induced by $S \rightarrow S'$).
\end{nn}

\begin{nn}
Let $f \in \Hom_\C(\bbm1, \bbm1)$ be a morphism and $C$ be an object in $\C$. The unit and counit isomorphisms give rise to morphisms $f^C_\lambda$ and $f^C_\varrho$ in the following way:
$$
f^C_\lambda = \lambda_C \circ (f \otimes C) \circ \lambda_C^{-1}: C \longrightarrow C
$$
and
$$
f^C_\varrho = \varrho_C \circ (C \otimes f) \circ \varrho_C^{-1}: C \longrightarrow C \ .
$$
In general, $f^C_\lambda \neq f^C_\varrho$, however equality holds if $C = \bbm1$. Note that $f_\lambda$ as well as $f_\varrho$ define natural transformations $\Id_\C \rightarrow \Id_\C$.
\end{nn}

\begin{defn}
Let $(\C, \otimes, \bbm1)$ be a monoidal $K$-category, and let $X$ be an object in $X$. For a morphism $f \in \Hom_\C(\bbm1, \bbm1)$, we let $\Delta^X(f) = f^X_\lambda - f^X_\varrho$ be the \textit{$X$-relative defect of $f$}. The set $Z_X(\C) = \{ f \in \Hom_\C(\bbm1, \bbm1) \mid \Delta^X(f) = 0 \}$ is called the \textit{$X$-relative center} of $\C$. The \textit{center} of $\C$ is $Z(\C) = Z_\bbm1(\C) = \Hom_\C(\bbm1, \bbm1)$.
\end{defn}

\begin{nn}
By \cite{Su02}, the center $Z(\C)$ of $\C$ is a commutative $K$-algebra. Clearly, $Z_X(\C)$ is a sub-$K$-algebra of $Z(\C)$ for all objects $X \in \C$. 

\begin{exas}
If $(\C, \otimes, \bbm1)$ is $(\Mod(A^\ev), \otimes_A, A)$ for some $K$-algebra $A$, and $M$ is an $A^\ev$-module, then the algebra homomorphism
$$
\gamma : Z_M(A) \longrightarrow Z_M(\Mod(A^\ev)), \ z \mapsto (a \mapsto za),
$$
is bijective. Thus the terminologies for relative centres agree. Similarly, if $(\C, \otimes, \bbm1)$ is $(\Mod(\Gamma), \otimes, K)$ for some $K$-bialgebra $\Gamma$, and $N$ is a $\Gamma$-module, then $Z_N(\Mod(\Gamma))$ is isomorphic to $K$.
\end{exas}

We have the following obvious properties of the $X$-relative center.

\begin{lem}
Let $X$ and $Y$ be objects in $\C$, and let $X'$ be a direct summand of $X$. The following statements hold true.
\begin{enumerate}[\rm(1)]
\item $Z_{X}(\C) \subseteq Z_{X'}(\C)$.
\item $Z_{X \oplus Y}(\C) = Z_X(\C) \cap Z_Y(\C)$.
\item $Z_X(\C) \cap Z_Y(\C) \subseteq Z_{X \otimes Y}(\C)$.\qed
\end{enumerate}
\end{lem}

\begin{nn}
Let $n \geqslant 0$ be an integer and let $X$ be an object in $\C$. Let $f \in Z_X(\C)$ be a morphism inside the $X$-relative center of $\C$. Further, let $S$ be an object in $\mathcal Ext^n_\C(\bbm1, X)$. We write $f^X$ for the morphism $f^X_\lambda = f^X_\varrho$.

For $n = 0$, the ``extension'' $S$ is simply a morphism $\bbm1 \rightarrow X$ and we let $\Omega_\C(S,f)$ be the trivial loop at $S \# f = S \circ f$. Assume that $n = 1$ and that $S$ is given by
$$
\xymatrix@C=20pt{0 \ar[r] & X \ar[r]^-{c_{1}} & C \ar[r]^-{c_{0}} & \bbm1 \ar[r] & 0} \ .
$$
Due to naturality, we obtain a morphism $S \rightarrow S$ of complexes:
$$
\xymatrix{
0 \ar[r] & X \ar[d]_-{f^X} \ar[r]^-{c_{1}} & C \ar@<1ex>[d]^-{f^C_\varrho} \ar@<-1ex>[d]_-{f^C_\lambda} \ar[r]^-{c_{0}} & \bbm1 \ar[d]^-f \ar[r] & 0 \ \ \\
0 \ar[r] & X \ar[r]^-{c_{1}} & C \ar[r]^-{c_{0}} & \bbm1 \ar[r] & 0 \ .
}
$$
The universal property of the pushout $X \oplus_X C$ of $f$ and $c_1$ yields the commutative diagram
$$
\xymatrix@R=25pt@C=30pt{
0 \ar[r] & X \ar[d]_-{f^X} \ar[r]^-{c_{1}} & C \ar@<1ex>@/^2.2pc/[dd]^-(0.3){f_\varrho^C} \ar[d] \ar@<-1ex>@/_2.2pc/[dd]_-(0.3){f_\lambda^C} \ar[r]^-{c_{0}} & \bbm1 \ar@{=}[d] \ar[r] & 0 \ \ \\
0 \ar[r] & X \ar@{=}[d] \ar[r]|(0.47){\hole} & X \oplus_X C \ar@<1ex>[d]^-{f'_\varrho} \ar@<-1ex>[d]_-{f'_\lambda} \ar[r]|(0.545){\hole} & \bbm1 \ar[r] \ar[d]^-f & 0 \ \ \\
0 \ar[r] & X \ar[r]^-{c_{1}} & C \ar[r]^-{c_{0}} & \bbm1 \ar[r] & 0 \ ,
}
$$
wherein the sandwiched admissible short exact sequence represents the Yoneda product $S \# f$ of $S$ with $f$. By pulling back $c_0$ and $f$, we obtain the commutative diagram
$$
\xymatrix@R=25pt@C=30pt{
0 \ar[r] & X \ar@{=}[d] \ar[r]^-{} & X \oplus_X C \ar@<1ex>@/^2.2pc/[dd]^-(0.7){f'_\varrho} \ar@<-1ex>@/_2.2pc/[dd]_-(0.7){f'_\lambda} \ar[r]^-{} \ar@<1ex>[d]^-{f''_\varrho} \ar@<-1ex>[d]_-{f''_\lambda} & \bbm1 \ar@{=}[d] \ar[r] & 0 \ \ \\
0 \ar[r] & X \ar@{=}[d] 
\ar[r]|(0.47)\hole & \bbm1 \times_\bbm1 C \ar[d] \ar[d] \ar[r]|(0.545){\hole} & \bbm1 \ar[r] \ar[d]^-f & 0 \ \ \\
0 \ar[r] & X \ar[r]^-{c_{1}} & C \ar[r]^-{c_{0}} & \bbm1 \ar[r] & 0 \ ,
}
$$
and hence a pair $F_\lambda, F_\varrho: S \# f \rightarrow f \# S$ of parallel morphisms in $\mathcal Ext^1_\C(\bbm1, X)$. They define a loop
$$
\xymatrix@C=20pt{S \# f \ar[r]^-{F_\lambda} & f \# S & S \# f \ar[l]_-{F_\varrho}}
$$
based at $S \# f$, that is, the desired loop $\Omega_\C(S,f)$.
\end{nn}

\begin{nn}
Now suppose that $n > 1$. Suppose further that $S$ is given as
$$
\xymatrix@C=22pt{0 \ar[r] & X \ar[r]^-{c_n} & C_{n-1} \ar[r]^-{c_{n-1}} & \cdots \ar[r]^-{c_{1}} & C_0 \ar[r]^-{c_{0}} & \bbm1 \ar[r] & 0} \ .
$$
Consider the morphisms $f_\lambda = f_\lambda^{C_{n-1}} = \lambda_{C_{n-1}} \circ (f \otimes C_{n-1}) \circ \lambda_{C_{n-1}}^{-1}$, $f_\varrho = f_\varrho^{C_{n-1}} = \varrho_{C_{n-1}} \circ (C_{n-1} \otimes f) \circ \varrho_{C_{n-1}}^{-1}$, $\tilde{f}_\lambda = f_\lambda^{C_{0}} = \lambda_{C_0} \circ (f \otimes C_0) \circ \lambda_{C_0}^{-1}$ and $\tilde{f}_\varrho = f_\varrho^{C_{0}} = \varrho_{C_0} \circ (C_0 \otimes f) \circ \varrho_{C_0}^{-1}$. They give rise to the following commutative diagram.
$$
\xymatrix@C=30pt{
0 \ar[r] & X \ar[d]_-{f^X} \ar[r]^-{c_n} & C_{n-1} \ar@<1ex>@/^2.5pc/[dd]^-(0.3){f_\varrho} \ar@<-1ex>@/_2.5pc/[dd]_-(0.3){f_\lambda} \ar[r]^-{c_{n-1}} \ar[d] & \cdots \ar[r]^-{c_{1}} & C_0 \ar[r]^-{c_{0}} \ar@{=}[d] & \bbm1 \ar@{=}[d] \ar[r] & 0\\
0 \ar[r] & X \ar@{=}[d] \ar[r]|(0.47){\hole} & X \oplus_X C_{n-1} \ar@<1ex>[d]^-{f'_\varrho} \ar@<-1ex>[d]_-{f'_\lambda} \ar[r]|(0.525){\hole} & \cdots \ar[r]^-{c_{1}} & C_0 \ar[r]^-{c_{0}} \ar@<1ex>@/^2.5pc/[dd]^-(0.7){\tilde{f}_\varrho} \ar@<-1ex>@/_2.5pc/[dd]_-(0.7){\tilde{f}_\lambda} \ar@<1ex>[d]^-{\tilde{f}'_\varrho} \ar@<-1ex>[d]_-{\tilde{f}'_\lambda} & \bbm1 \ar[r] \ar@{=}[d] & 0\\
0 \ar[r] & X \ar@{=}[d] \ar[r]^-{c_n} & C_{n-1} \ar@{=}[d] \ar[r]^-{c_{n-1}} & \cdots \ar[r]|(0.41){\hole} & \bbm1 \times_\bbm1 C_0 \ar[r]|(0.62){\hole} \ar[d] & \bbm1 \ar[r] \ar[d]^-f & 0\\
0 \ar[r] & X \ar[r]^-{c_n} & C_{n-1} \ar[r]^-{c_{n-1}} & \cdots \ar[r]^-{c_{1}} & C_0 \ar[r]^-{c_{0}} & \bbm1 \ar[r] & 0
}
$$
Yet again, we obtain two parallel morphisms $F_\lambda, F_\varrho: S \# f \rightarrow f \# S$ in $\mathcal Ext^n_\C(\bbm1, X)$ and hence acquire the loop $\Omega_\C(S,f)$ based at $S \# f$.
\end{nn}

\begin{nn}
The above construction yields a map
$$
\langle -,- \rangle = \langle -,- \rangle_X : \Ext^n_\C(\bbm1, X) \times Z_X(\C) \longrightarrow \Ext^{n-1}_\C(\bbm1, X)
$$
by $\langle \alpha, f \rangle = u_\C^{-1} \Omega_\C(S,f)$, where $u_\C$ is the isomorphism described in \ref{nn:uC} (see also Theorem \ref{thm:retakh}) and $S$ is an extension representing $\alpha$. The following is a question we will not be able to answer in general.

\begin{quest}
Does the map
$$
\langle -,- \rangle_X : \Ext^\ast_\C(\bbm1, X) \times Z_X(\C) \longrightarrow \Ext^{\ast - 1}_\C(\bbm1, X)
$$
turn $\Ext^\ast_\C(\bbm1, X)$ into a right Gerstenhaber module over $Z_X(\C)$?
\end{quest}

However, for $\C = \Mod(A^\ev)$ we are going to prove the following main result.
\end{nn}

\begin{thm}\label{thm:mainthm}
Let $A$ be a $K$-algebra and $M$ be an $A^\ev$-module. The following diagram commutes for $n = 0,1$.
$$
\xymatrix{
\HH^n(A,M) \times Z_M(A) \ar[r]^-{[-,-]} \ar[d]_-{\cong} & \HH^{n-1}(A,M) \ar[d]^-{\cong}\\
\Ext^n_{A^\ev}(A,M) \times Z_{M}(\Mod(A^\ev)) \ar[r]^-{\langle -, -\rangle} & \Ext^{n-1}_{A^\ev}(A,M)
}
$$
It also commutes for $n > 1$ provided that $A$ is projective as a $K$-module.
\end{thm}

As we will elaborate on how the behaviour of an exact monoidal category $(\C, \otimes, \bbm1)$ influences the associated maps $\langle -, -\rangle_X$ for $X \in \C$, the following consequences will be imminent.

\begin{cor}\label{cor:funccompat}
Let $A$ and $B$ be $K$-algebras, and assume that $A$ and $B$ are projective over $K$. Let $\mathfrak X: (\Mod(A^\ev), \otimes_A, A) \longrightarrow (\Mod(B^\ev), \otimes_B, B)$ be an exact and almost $($co$)$strong monoidal functor. Then there are induced homomorphisms $\delta_M: Z_M(A) \rightarrow Z_{\mathfrak X M}(B)$ and $\mathfrak X^\ast : \HH^\ast(A,M) \rightarrow \HH^\ast(B,\mathfrak X M)$ which render the following diagram commutative.
$$
\xymatrix@C=30pt{
\HH^n(A,M) \times Z_M(A) \ar[r]^-{[-,-]_M} \ar[d]_-{\mathfrak X^n \times \delta_M} & \HH^{n-1}(A, M) \ar[d]^-{\mathfrak X^{n-1}}\\
\HH^n(B,\mathfrak X M) \times Z_{\mathfrak X M}(B) \ar[r]^-{[-,-]_{\mathfrak X M}} & \HH^{n-1}(B, \mathfrak X M)
}
$$
In particular, if $B \cong \End_{A}(P)^\op$ for some $A$-progenerator $P$ $($that is, $A$ and $B$ are Morita equivalent$)$, and if $N$ is the $B^\ev$-module $\Hom_{A^\ev}(P \otimes \Hom_A(P,A), M)$, then $\HH^\ast(A,M)$ and $\HH^\ast(B, N)$ will be isomorphic as right Gerstenhaber modules over $Z_M(A) \cong Z_{N}(B)$.
\end{cor}

\begin{cor}\label{cor:vanish}
Let $A$ be a $K$-algebra which is projective over $K$. Consider the following statements.
\begin{enumerate}[\rm(1)]
\item\label{cor:vanish:0} $\HH^0(Z(A),M) = M$ for all $M \in \Mod(A^\ev)$.
\item\label{cor:vanish:1} $Z(A) = Z_{M}(A)$ for all $M \in \Mod(A^\ev)$.
\item\label{cor:vanish:2} $Z(A) = Z_{A \otimes A}(A)$.
\item\label{cor:vanish:3} $(\Mod(Z(A)^\ev), \otimes_{Z(A)}, Z(A))$ is braided monoidal.
\item\label{cor:vanish:4} $[-,-]_M$ vanishes for all $M \in \Mod(A^\ev)$.
\end{enumerate}
Then one has the implications
$$
(\ref{cor:vanish:0}) \ \Longleftrightarrow \ (\ref{cor:vanish:1}) \ \Longleftrightarrow \ (\ref{cor:vanish:2}) \ \Longleftrightarrow \ (\ref{cor:vanish:3}) \ \Longrightarrow \ (\ref{cor:vanish:4})
$$
amongst them.
\end{cor}

Before turning ourselves to the proofs of the above results, let us mention that if $K$ is a field, $K \rightarrow Z(A)$ being an epimorphism is equivalent to it being an isomorphism, by \cite[Cor.\,1.2]{Si67}, and hence $(\ref{cor:vanish:3}) \, \Longrightarrow \, (\ref{cor:vanish:4})$ is automatic. Moreover, the implication $(\ref{cor:vanish:4}) \, \Longrightarrow \, (\ref{cor:vanish:3})$ will not hold true in general. Indeed, if, for instance, $A = L \supseteq K$ is a separabel field extension, (\ref{cor:vanish:4}) will surely be satisfied (see \cite{Ho45}), whereas $\mu: L \otimes_K L \rightarrow L$ cannot be an isomorphism unless $[L:K] = 1$.

%%%%%%%%%%%%%%%%%%%%%%%%%%%%%%%%
%%% Proofs
%%%%%%%%%%%%%%%%%%%%%%%%%%%%%%%%

\section{Proofs}\label{sec:proofs}

\subsection{Compatibility results} In what follows, we will assume that our exact and monoidal categories are tensor $K$-categories and closed under kernels of epimorphisms. Let $(\C, \otimes, \bbm1) = (\C, \otimes_\C, \bbm1_\C)$ be such a category. For two objects $W, X$ in $\C$ and an integer $n \geqslant 1$, we let $\mathsf E^n_\C(W,X)$ be the full subcategory of $\C$ obtained as follows: An object $E \in \C$ belongs to $\mathsf E^n_\C(W,X)$ if, and only if, there is an admissible $n$-extension $0 \rightarrow X \rightarrow E_{n-1} \rightarrow \cdots \rightarrow E_0 \rightarrow W \rightarrow 0$ such that $E \cong E_i$ for some $0 \leqslant i \leqslant n-1$. Observe that $\mathsf E^n_\C(W,X) = \C$ for $n \geqslant 3$, since
$$
\xymatrix@C=18pt{
0 \ar[r] & X \ar[r]^-{\mathrm{can}} & X \oplus E \ar[r]^-{\mathrm{can}} & E \ar[r]^-0 & \cdots \ar[r]^0 & W \ar@{=}[r] & W \ar[r] & 0
}
$$
is an admissible $n$-extension for every object $E \in \C$.

\begin{lem}\label{lem:EnC}
For an object $X \in \C$ and an integer $n \geqslant 1$ consider the following statements.
\begin{enumerate}[\rm(1)]
\item\label{lem:EnC:0} $Z_X(\C) \subseteq Z_E(\C)$ for all objects $E$ in $\C$.
\item\label{lem:EnC:1} $Z_X(\C) \subseteq Z_E(\C)$ for all objects $E$ in $\E^n_\C(\bbm1,X)$.
\item\label{lem:EnC:2} The $n$-th component map of $\langle -, -\rangle_X$,
$$
\Ext^n_\C(\bbm1, X) \times Z_X(\C) \longrightarrow \Ext^{n-1}_\C(\bbm1, X),
$$
is constantly zero.
\end{enumerate}
Then the implications $(\ref{lem:EnC:0}) \, \Longrightarrow \, (\ref{lem:EnC:1}) \, \Longrightarrow \, (\ref{lem:EnC:2})$ hold. Further, $(\ref{lem:EnC:2}) \, \Longrightarrow \, (\ref{lem:EnC:1})$ if $n = 1$, and $(\ref{lem:EnC:1}) \, \Longrightarrow \, (\ref{lem:EnC:0})$ if $n > 1$.
\end{lem}

\begin{proof}
The implication $(\ref{lem:EnC:0}) \, \Longrightarrow \, (\ref{lem:EnC:1})$ is trivial. If $f_\lambda^E = f_\varrho^E$ for all $f \in \Hom_\C(\bbm1, \bbm1)$ and $E \in \E^n_\C(\bbm1, X)$, then the morphisms $F_\lambda, F_\varrho: S \# f \rightarrow f \# S$ defining the loop $\Omega_\C(S,f)$ for $S \in \mathcal Ext^n_\C(\bbm1, X)$ agree, and hence $\Omega_\C(S,f)$ will be (homotopically equivalent to) the trivial loop. Thus (\ref{lem:EnC:1}) implies (\ref{lem:EnC:2}).

Let $f \in Z_X(\C)$ and $S$ be an admissible $1$-extension $0 \rightarrow X \rightarrow E \rightarrow \bbm1 \rightarrow 0$. The category $\mathcal Ext^1_\C(\bbm1, X)$ is a groupoid, so $\mathsf G(\mathcal Ext^1_\C(\bbm1, X)) \cong \mathcal Ext^1_\C(\bbm1, X)$. Hence a loop $S' \leftarrow T \rightarrow S'$ in $\mathcal Ext^1_\C(\bbm1, X)$ will be homotopically equivalent to the trivial loop if, and only if, the two arrows $S' \leftarrow T \rightarrow S'$ are the same. Thus, $\Omega_\C(S,f)$ is homotopically equivalent to the trivial loop if, and only if, $F_\lambda = F_\varrho$. By construction, this is if, and only if, $f_\lambda^E = f_\varrho^E$. Therefore (\ref{lem:EnC:2}) implies (\ref{lem:EnC:1}) if $n = 1$.

Finally, the canonical sequence $0 \rightarrow X \rightarrow X \oplus E \rightarrow \bbm1 \oplus E \rightarrow \bbm1 \rightarrow 0$ is admissible exact for all objects $E \in \C$. Thus if (\ref{lem:EnC:1}) is valid for $n = 2$, then $Z_X(\C) \subseteq Z_{\bbm1 \oplus E} = Z(\C) \cap Z_E(\C) = Z_E(\C)$. Therefore the second item implies the first if $n > 1$.
\end{proof}

The (Drinfel'd) center of the monoidal category $\C$, as introduced in \cite{JoSt91}, is related to our definition, as we will notice shortly. Let us recall this object, with slightly changed terminology.
\end{nn}

\begin{defn}
Let $\mathsf U \subseteq \C$ be a monoidal subcategory. The \textit{$\mathsf U$-restricted monoidal center} of $\C$ is the following category, denoted by $\mathcal Z(\mathsf U, \C)$. Its objects are pairs $(X,a)$, where $X \in \Ob \C$ and $a : (- \otimes X) \rightarrow (X \otimes -)$ is a natural isomorphism of functors $\mathsf U \rightarrow \C$, such that $a_\bbm1 = \varrho_X^{-1} \circ \lambda_X$ and $a_{U \otimes V} = (V \otimes a_U) \circ (a_V \otimes U)$ for all $U, V \in \Ob \mathsf U$. A morphism $(X,a) \rightarrow (X',a')$ in $\mathcal Z(\mathsf U, \C)$ is a morphism $f: X \rightarrow X'$ with $(f \otimes U) \circ a_U = a'_U \circ (U \otimes f)$ for all $U \in \Ob \mathsf U$.
\end{defn}

$\mathcal Z(\C, \C)$ is a braided tensor $K$-category, with tensor functor $(X,a) \otimes (X',a') = (X \otimes X', (X \otimes a')\circ(a \otimes X'))$ and braiding $\gamma_{(X,a),(X',a')} = a_{X'}$.

\begin{lem}\label{lem:Xrelcen}
For an object $X \in \C$ the following statements are equivalent.
\begin{enumerate}[\rm(1)]
\item\label{lem:Xrelcen:1} $Z(\C) = Z_{X}(\C)$.
\item\label{lem:Xrelcen:2} The isomorphism $\gamma_X = \varrho_X^{-1} \circ \lambda_X$ gives rise to an isomorphism between the functors
$$
(- \otimes -),\, (- \otimes - ) \circ T : \add(\bbm1) \times \add(X) \longrightarrow \C,
$$
where $T$ denotes the twist functor $T : \C \times \C \rightarrow \C \times \C$, $T(C,D) = (D,C)$.
\item\label{lem:Xrelcen:3} There is a natural transformation $a : (- \otimes X) \rightarrow (X \otimes -)$ between functors $\add(\bbm1) \rightarrow \C$ such that $(X,a)$ belongs to $\mathcal Z(\add(\bbm1), \C)$.
\end{enumerate}
\end{lem}

\begin{proof}
If (\ref{lem:Xrelcen:1}) holds, then for every $f \in \Hom_\C(\bbm1, \bbm1)$ the diagram
$$
\xymatrix@C=25pt{
\bbm1 \otimes X \ar[d]_-{\gamma_X} \ar[r]^-{\lambda_X} & X \ar@{=}[d] \ar[r]^-{f_\lambda^X} & X \ar@{=}[d] \ar[r]^-{\lambda^{-1}_X} & \bbm1 \otimes X \ar[d]^-{\gamma_X}\\
X \otimes \bbm1 \ar[r]^-{\varrho_X} & X \ar[r]^-{f_\varrho^X} & X \ar[r]^-{\varrho^{-1}_X} & X \otimes \bbm1
}
$$
commutes. Thus, $\gamma_X \circ (X \otimes f) = (f \otimes X) \circ \gamma_X$. If now $U = \bigoplus_i \bbm1$, $Y = \bigoplus_j X$ are finite direct sums of copies of $\bbm1$ and $X$, and $U' \subseteq U$, $Y' \subseteq Y$ are direct summands, then put
\begin{equation}\label{eq:gammaex}
\begin{aligned}
\gamma_{\bbm1, Y} = \bigoplus_j \gamma_X \, , &\qquad \gamma_{\bbm1, Y'} = (Y \xrightarrow{\mathrm{can}} Y') \circ \gamma_Y \circ (Y' \xrightarrow{\mathrm{can}} Y)\, ,\\
\gamma_{U, Y'} = \bigoplus_{i} \gamma_{\bbm1, Y'}\, , &\qquad \gamma_{U', Y'} = (U \xrightarrow{\mathrm{can}} U') \circ (\gamma_{U,Y'}) \circ (U' \xrightarrow{\mathrm{can}} U)
\end{aligned}
\end{equation}
in order to obtain the desired functorial isomorphisms. Conversely, (\ref{lem:Xrelcen:2}) implies (\ref{lem:Xrelcen:1}) by similar arguments. The implication $(\ref{lem:Xrelcen:2}) \,\Longrightarrow\, (\ref{lem:Xrelcen:3})$ is obvious, and its convers follows by extending $a$ to direct summands of direct sums of $X$ as done in (\ref{eq:gammaex}).
\end{proof}

\begin{prop}\label{prop:Xrelcennat}
Consider the following statements.
\begin{enumerate}[\rm(1)]
\item\label{prop:Xrelcennat:1}$Z(\C) = Z_{X}(\C)$ for all $X \in \C$.
\item\label{prop:Xrelcennat:2} The isomorphisms $\gamma_X = \varrho_X^{-1} \circ \lambda_X$, for $X \in \C$, define an isomorphism between the functors
$$
(- \otimes -),\, (- \otimes - ) \circ T : \add(\bbm1) \times \C \longrightarrow \C,
$$
where $T$ denotes the twist functor $T : \C \times \C \rightarrow \C \times \C$, $T(C,D) = (D,C)$.
\item\label{prop:Xrelcennat:2a} For all $X \in \C$, there is a natural transformation $a : (- \otimes X) \rightarrow (X \otimes -)$ between functors $\add(\bbm1) \rightarrow \C$ such that $(X,a)$ belongs to $\mathcal Z(\add(\bbm1),\C)$.

\item\label{prop:Xrelcennat:3} $\langle -, - \rangle_X \equiv 0$ for all $X \in \C$.
\end{enumerate}
Then the implications $(\ref{prop:Xrelcennat:1}) \, \Longleftrightarrow \, (\ref{prop:Xrelcennat:2}) \,\Longleftrightarrow \, (\ref{prop:Xrelcennat:2a}) \, \Longrightarrow \, (\ref{prop:Xrelcennat:3})$ hold true. Thus $\langle -,- \rangle_X \equiv 0$ for all $X \in \C$ if $Z(\C) = K \id_\bbm1$.
\end{prop}

\begin{proof}
The equivalence of (\ref{prop:Xrelcennat:1}), (\ref{prop:Xrelcennat:2}) and (\ref{prop:Xrelcennat:2a}) follows from Lemma \ref{lem:Xrelcen}, whereas (\ref{prop:Xrelcennat:1}) implies (\ref{prop:Xrelcennat:3}) by Lemma \ref{lem:EnC}.
\end{proof}

\begin{nn}
Let $\C$ and $\D$ be exact $K$-categories, and let $\mathfrak X: \C \rightarrow \D$ be an exact functor. As $\mathfrak X$ takes admissible exact sequences to admissible exact sequences, $\mathfrak X$ gives rise to a functor $\mathsf G(\mathcal Ext^n_\C(C,D)) \rightarrow \mathsf G(\mathcal Ext^n_\D(\mathfrak X C, \mathfrak X D)$, see Section \ref{sec:funda}, and thus to a group homomorphism $\pi_1(\mathcal Ext^n_\C(C,D),S) \rightarrow \pi_1(\mathcal Ext^n_\D(\mathfrak X C,\mathfrak X D),\mathfrak X (S))$ for each admissible extension $S \in \mathcal Ext^n_\C(C,D)$. By \cite[Lem.\,3.2.6]{He14b}, this map renders
\begin{equation}\label{eq:commdiaggr}
\begin{aligned}
\xymatrix{
\Ext^n_\C(C,D) \ar[r]^-{u_\C} \ar[d]_-{\mathfrak X^n} & \pi_1(\mathcal Ext^n_\C(C,D),S) \ar[d] \\
\Ext^n_\D(\mathfrak X C, \mathfrak X D) \ar[r]^-{u_\D} & \pi_1(\mathcal Ext^n_\D(\mathfrak X C,\mathfrak X D), \mathfrak X(S))
}
\end{aligned}
\end{equation}
commutative, where $\mathfrak X^n$ sends the equivalence class of an extension $T$ to the equivalence class of the extension $\mathfrak X(T)$.
\end{nn}

\begin{prop}\label{prop:functcomp}
Assume that $(\C, \otimes_\C, \bbm1_\C)$ and $(\D, \otimes_\D, \bbm1_\D)$ are exact and monoidal $K$-categories. Let $\mathfrak X: (\C, \otimes_\C, \bbm1_\C) \rightarrow (\D, \otimes_\D, \bbm1_\D)$ be an exact and almost strong monoidal functor. Let $X$ be an object in $\C$ and $Y = \mathfrak X(X)$.
\begin{enumerate}[\rm(1)]
\item The functor $\mathfrak X$ induces a $K$-algebra homomorphism $\delta_X: Z_X(\C) \rightarrow Z_{Y}(\D)$. It is an injection $($surjection$)$ if, and only if, the functor $\mathfrak X$ defines an injection $($surjection$)$ $\Hom_\C(\bbm1_\C, \bbm1_\C) \rightarrow \Hom_\D(\mathfrak X \bbm1_\C, \mathfrak X \bbm1_\C)$.
\item\label{prop:functcomp:2} The functor $\mathfrak X$ gives rise to a graded map $\mathfrak X^\ast : \Ext^\ast_\C(\bbm1_\C, X) \rightarrow \Ext^\ast_\D(\bbm1_\D, Y)$ such that
$$
\xymatrix@C=35pt{
\Ext^n_\C(\bbm1_\C, X) \times Z_X(\C) \ar[d]_-{\mathfrak X^{n} \times \delta_X} \ar[r]^-{\langle -, -\rangle_X} & \Ext^{n-1}_\C(\bbm1_\C, X) \ar[d]^-{\mathfrak X^{n-1}}\\
\Ext^n_\D(\bbm1_\D, Y) \times Z_{Y}(\D) \ar[r]^-{\langle -, -\rangle_Y} & \Ext^{n-1}_\D(\bbm1_\D, Y)
}
$$
is commutative for all integers $n \geqslant 1$. The graded map $\mathfrak X^\ast$ is bijective, if $\mathfrak X$ is an equivalence of $K$-categories.
\end{enumerate}
\end{prop}

\begin{proof}
The map $\delta_X$ is induced by restricting $\delta: Z(\C) \rightarrow Z(\D)$, $\delta(f) = \phi_0^{-1} \circ \mathfrak X(f) \circ \phi_0$. Indeed, for $f \in \Hom_\C(\bbm1_\C, \bbm1_\C)$, $C \in \Ob \C$ and $D = \mathfrak X(C)$, the diagram
$$
\xymatrix@R=14pt@C=27pt{
& \bbm1_\D \otimes_\D D \ar[r]^-{\phi_0 \otimes_\D D} \ar[dd]_-{\lambda_D^\D} & \mathfrak X \bbm1_\C \otimes_\D D \ar[r]^-{\mathfrak X(f) \otimes_\D D} \ar[dd]_-{\phi_{\bbm1_\C, C}} & \mathfrak X \bbm1_\C \otimes_\D D \ar[dd]^-{\phi_{\bbm1_\C, C}} \ar[r]^-{\phi_0^{-1} \otimes_\D D} & \bbm1_\D \otimes_\D D \ar[dd]^-{\lambda^\D_D} \ar[dr]^-{\lambda^\D_D} &\\
D \ar[ur]^-{(\lambda_D^\D)^{-1}} \ar[dr]_-{\id_D} & & & & & D\\
& D \ar[r]^-{\mathfrak X(\lambda^\C_C)^{-1}} & \mathfrak X (\bbm1_\C \otimes_\C C) \ar[r]^-{\mathfrak X(f \otimes_\C C)} & \mathfrak X(\bbm1_\C \otimes_\C C) \ar[r]^-{\mathfrak X(\lambda^\C_C)} & D\ar[ur]_-{\id_D} &
}
$$
commutes, that is, $\delta(f)^{D}_\lambda = \delta(f^C_\lambda)$. Similarly, we have $\delta(f)^{D}_\varrho = \delta(f^C_\varrho)$, so that $f \in Z_X(\C)$ will imply $\delta(f) \in Z_{Y}(\D)$. Moreover, $\delta_X$ is an algebra homomorphism since $\phi_0 \circ \phi_0^{-1} = \id_{\bbm1_\D}$. By definition, $\delta_X$ will be injective/surjective, if $\mathfrak X$ fulfils the requirements listed.

As for item (\ref{prop:functcomp:2}), the point is that since $\mathfrak X$ is exact, it commutes with pushouts along admissible monomorphisms and pullbacks along admissible epimorphisms. Thus the induced functor $\mathfrak X : \mathcal Ext^n_\C(\bbm1_\C,X) \rightarrow \mathcal Ext^n_\D(\bbm1_\D, \mathfrak X(X))$ sending an extension $0 \rightarrow X \rightarrow E_{n-1} \rightarrow \cdots \rightarrow E_0 \rightarrow \bbm1_\C \rightarrow 0$ to  $0 \rightarrow \mathfrak X(X) \rightarrow \mathfrak X(E_{n-1}) \rightarrow \cdots \rightarrow \mathfrak X(\bbm1_\C) \cong \bbm1_\D \rightarrow 0$ satisfies $\mathfrak X(S\#f) = \mathfrak X(S)\#\delta(f)$ and $\mathfrak X(f\#S) = \delta(f)\#\mathfrak X(S)$ for all $S \in \mathcal Ext^n_\C(\bbm1_\C, X)$, $f \in \Hom_\C(\bbm1_\C, \bbm1_\C)$. Now, as we already noticed,
$$
\delta(f_\lambda^C) = \delta(f)_\lambda^{\mathfrak X(C)} \quad \text{and} \quad \delta(f_\varrho^C) = \delta(f)_\varrho^{\mathfrak X(C)},
$$
so that $\mathfrak X$ will take the morphisms of complexes that define $\Omega_\C(S,f)$ to those that define $\Omega_\D(\mathfrak X, \delta_X(f))$ (for $S \in \mathcal Ext^n_\C(\bbm1_\C, X)$, $f \in Z_X(\C)$). To summarise,
\begin{equation}\label{eq:loops}
\mathfrak X \Omega_\C(S,f) = \Omega_\D(\mathfrak X(S),\delta_X(f)).
\end{equation}
Observe that we actually have strict \textit{equality} not only homotopy equivalence. We conclude by
\begin{align*}
\mathfrak X^{n-1}(\langle \alpha, f\rangle_X) &= \mathfrak X^{n-1}(u_\C^{-1}\Omega_\C(S,f)) && \text{(by definition of $\langle -,- \rangle_X$)}\\
&= u_\D^{-1} \mathfrak X\Omega_\C(S,f)) && \text{(by the commutativity of (\ref{eq:commdiaggr}))} \\
&= u_\D^{-1} \Omega_\D(\mathfrak X(S),\delta_X(f)) && \text{(by (\ref{eq:loops}))}\\
&= \langle \mathfrak X^n(\alpha), \delta_X(f)\rangle_Y && \text{(by definition of $\langle -,- \rangle_Y$)}
\end{align*}
for all $\alpha = [S] \in \Ext^n_\C(\bbm1_\C, X)$ and $f \in Z_X(\C)$.
\end{proof}

\begin{rem}
Proposition \ref{prop:functcomp} remains valid if one replaces the term \textit{almost strong monoidal} by \textit{almost costrong monoidal}. However, the maps $\delta_X$ and $\mathfrak X^\ast$ will appear in a slightly different shape.
\end{rem}

\subsection{The loop $\Omega_\C(S,f)$ for modules}
For the remainder of this section, we fix a $K$-algebra $A$ (not necessarily projective over $K$) and an $A^\ev$-module $M$. Further, $\otimes$ will always stand for $\otimes_K$. Let us explicitly describe the (morphisms occurring in the) loops $\Omega_A(S,f) = \Omega_{\C}(S,f)$ where $\C = \Mod(A^\ev)$, $S$ is an $n$-extension of $A$ by $M$ in $\Mod(A^\ev)$ and $f \in Z_M(A) \subseteq \Hom_{A^\ev}(A,A)$. In what follows, we will abuse notation, and also write $f$ for the map $f^M = f_\lambda^M = f_\varrho^M : M \rightarrow M$. 

To begin with, assume that $S$ is a short exact sequence
$$
\xymatrix@C=18pt{
0 \ar[r] & M \ar[r]^-{i} & E \ar[r]^-{p} & A \ar[r] & 0
}\, .
$$
Recall that the pushout $M \oplus_M E$ of $(f, i)$ and the pullback $A \times_A E$ of $(f, p)$ can be expressed as
$$
M \oplus_M E = \frac{M \oplus E}{\{(f(m), -i(m)) \mid m \in M\}} = \Coker(f \oplus (-i)) \, ,
$$
$$
A \times_A E = \{(a,e) \in A \oplus E \mid f(a) = p(e)\} = \Ker(f + (-p))\, .
$$
The homomorphisms $f''_\lambda, f''_\varrho: M \oplus_M E \rightarrow A \times_A E$ which define $\Omega_A(S,f)$ are induced by the maps
$$
\zeta''_\lambda: M \oplus E \longrightarrow A \oplus E, \ (m,e) \mapsto (p(e), f(1)e + i(m))
$$
and
$$
\zeta''_\varrho: M \oplus E \longrightarrow A \oplus E, \ (m,e) \mapsto (p(e), ef(1) + i(m))\, .
$$
Now let $S$ be in $\mathcal Ext^n_{A^\ev}(A, M)$ for some integer $n > 1$. Suppose that $S$ has the following shape:
$$
\xymatrix@C=18pt{
0 \ar[r] & M \ar[r]^-{d_{n}} & E_{n-1} \ar[r]^-{d_{n-1}} & \cdots \ar[r]^-{d_1} & E_0 \ar[r]^-{d_0} & A \ar[r] & 0
}\, .
$$
The defining homomorphisms $f'_\lambda, f'_\varrho: M \oplus_M E_{m-1} \rightarrow E_{m-1}$, $\tilde{f}'_\lambda, \tilde{f}'_\varrho: E_0 \rightarrow A \times_A E_0$ of the loop $\Omega_\C(S,z)$ are induced by the following maps:
$$
\zeta'_\lambda: M \oplus E_{n-1} \longrightarrow E_{n-1}, \ (m,e) \mapsto f(1)e + d_n(m)\, ,
$$
$$
\zeta'_\varrho: A \oplus E_{n-1} \longrightarrow E_{n-1}, \ (m,e) \mapsto ef(1) + d_n(m)\, \,
$$
and
$$
\tilde{\zeta}'_\lambda: E_0 \longrightarrow A \oplus E_0, \ e \mapsto (d_0(e), f(1)e)\, ,
$$
$$
\tilde{\zeta}'_\varrho: E_0 \longrightarrow A \oplus E_0, \ e \mapsto (d_0(e), ef(1))\, .
$$
In all other degrees, the maps $E_i \rightarrow E_i$ are given by $f_\lambda^{E_i}$ and $f_\varrho^{E_i}$ respectively (for $i = 1, \dots, n-2$). Recall that, for any $n$, the maps $M \rightarrow M \oplus_M E_{n-1} \rightarrow E_{n-2}$ in $S \# f$ are induced by $m \mapsto (m,0)$ and $(m,e) \mapsto d_{n-1}(e)$, whereas the maps $E_1 \rightarrow A \times_A E_0 \rightarrow A$ in $f \# S$ are induced by $e \mapsto (0,d_1(e))$ and $(a,e) \mapsto e$.

\begin{nn}
Let us fix a projective resolution $\mathbb P A \rightarrow A \rightarrow 0$ of $A$ over $A^\ev$. Assume that it is given as
$$
\xymatrix@C=20pt{
\cdots \ar[r]^-{\pi_{n+2}} & P_{n+1} \ar[r]^-{\pi_{n+1}} & P_{n} \ar[r]^-{\pi_{n}} & P_{n-1} \ar[r]^-{\pi_{n-1}} & \cdots \ar[r]^-{\pi_1} & P_0 \ar[r]^-{\pi_0} & A \ar[r] & 0
}.
$$
Since $A \otimes A$ is projective as an $A^\ev$-module, and the multiplication map ${\mu: A \otimes A} \rightarrow A$ defines a surjective $A^\ev$-homomorphism, we may, and will, assume that $P_0 = A \otimes A$ and $\pi_0 = \mu$.

The multiplication map gives rise to the fundamental short exact sequence
\begin{equation}\tag{$\diamond$}\label{eq:fundexact}
\xymatrix@C=18pt{
0 \ar[r] & \Omega^1_A \ar[r]^-\iota & A \otimes A \ar[r]^-\mu \ar[r] & A \ar[r] & 0
}
\end{equation}
which will play a key role in the considerations below. The following lemmas are classical, see \cite[Chap.\,III, \S 10]{Bou89} or \cite[Prop.\,2.5]{CuQu95}, but will be stated for completeness.
\end{nn}

\begin{lem}\label{lem:omegagen}
The module $\Omega^1_A$ is, as a left and a right $A$-module, generated by the elements $\mathrm d a = a \otimes 1 - 1 \otimes a$ for $a \in A$.
\end{lem}

\begin{proof}
Clearly, $\mu( \mathrm d a ) = 0$ for all $a \in A$. If
$$
\mu \left( \sum_{i=1}^n{a_i \otimes b_i} \right) = 0 \quad \text{(for $a_1, \dots, a_n, b_1, \dots, b_n \in A$)},
$$
then $$\sum_{i=1}^n a_i \otimes b_i = \sum_{i=1}^n (\mathrm d a_i) b_i = \sum_{i=1}^n -a_i \mathrm d b_i.$$
\end{proof}

\begin{lem}
The map $\mathrm d : A \rightarrow \Omega_A^1, \ \mathrm d a = a \otimes 1  - 1 \otimes a$ is a $K$-linear derivation into the $A^\ev$-module $\Omega^1_A$ and has the following universal property: For every $A^\ev$-module $M$ and ever $K$-linear derivation $D : A \rightarrow M$ there exists a unique $A^\ev$-linear map $\overline D : \Omega^1_A \rightarrow M$ such that $\overline D \circ \mathrm d = D$. Moreover, the assignments $D \mapsto \overline D$ and $f \mapsto f \circ \mathrm d$ define mutually inverse isomorphisms between $\Der_K(A,M)$ and $\Hom_{A^\ev}(\Omega_A^1, M)$. They identify $\Inn_K(A,M)$ with $\Im \Hom_{A^\ev}(\iota, M)$.
\end{lem}

\begin{proof}
It is straightforward to check that $\mathrm d$ is a derivation as claimed. Let $M$ be an $A^\ev$-module with right module structure map $\mu_M = \mu^r_M : M \otimes A \rightarrow M$. If $D: A \rightarrow M$ is any $K$-linear derivation then, $\overline D = \mu_M \circ (D \otimes A) \circ \iota$ is $A^\ev$-linear. In fact, if $a, a', a'' \in A$, then
$$
\overline D(a' (\mathrm d a) a'') = ( D(a'a) - D(a') a ) a'' = a' D(a) a''.
$$
Apparently, $\overline D \circ \mathrm d = D$, and $\overline D$ is the unique $A^\ev$-linear map with this property. Hence we obtain isomorphisms as claimed.
\end{proof}

\begin{lem}
Let $M$ be an $A^\ev$-module. Then, as $K$-modules, $\HH^0(A,M)$ is isomorphic to $\Hom_{A^\ev}(A,M)$ and $\HH^1(A,M) = \Out_K(A,M)$ is isomorphic to $\Ext^1_{A^\ev}(A,M)$.
\end{lem}

\begin{proof}
The left exactness of $\Hom_{A^\ev}(-,M)$ forces
$$
\xymatrix@C=18pt{
0 \ar[r] & \Hom_{A^\ev}(A,M) \ar[r]^-{\mu^\ast} & \Hom_{A^\ev}(A \otimes A, M) \ar[r]^-{\beta_1^\ast} & \Hom_{A^\ev}(A \otimes A \otimes A, M)
}
$$
to be exact ($\beta_1: A^{\otimes 3} \rightarrow A^{\otimes 2}$ denotes the first differential in $\mathbb BA$). Thus,
$$
\HH^0(A,M) = H^0 \Hom_{A^\ev}(\mathbb B A, M) \cong \Ker \Hom_{A^\ev}(b_1,M) \cong \Hom_{A^\ev}(A,M).
$$
The fundamental exact sequence (\ref{eq:fundexact}) yields the exact sequence
$$
\xymatrix@C=12pt{
0 \ar[r] & \Hom_{A^\ev}(A,M) \ar[r]^-{\mu^\ast} & \Hom_{A^\ev}(A \otimes A, M) \ar[r]^-{\iota^\ast} & \Hom_{A^\ev}(\Omega^1_A, M) \ar[r] & \Ext^1_{A^\ev}(A,M) \ar[r] & 0
}
$$
which, when combined with the preceding lemma, gives
$$
\Out_K(A,M) = \frac{\Der_K(A,M)}{\Inn_K(A,M)} \cong \frac{\Hom_{A^\ev}(\Omega^1_A,M)}{\Im(\iota^\ast)} \cong \Ext^1_{A^\ev}(A,M).
$$
Hence the lemma is established.
\end{proof}

\begin{nn}
Let us elaborate further on the surjection
\begin{equation}\label{eq:surj}
\xymatrix@C=20pt{
\Der_K(A,M) \ar[r]^-{\mathrm{can}} & \displaystyle \frac{\Der_K(A,M)}{\Inn_K(A,M)} \ar[r]^-\sim & \Ext^1_{A^\ev}(A,M)
}.
\end{equation}
It sends a $K$-linear derivation $D$ to the equivalence class of the lower sequence in the pushout diagram below.
$$
\xymatrix{
0 \ar[r] & \Omega^1_A \ar[r]^-\iota \ar[d]_-{\overline{D}} & A \otimes A \ar[r]^-\mu \ar[d] & A \ar@{=}[d] \ar[r] & 0 \\
0 \ar[r] & M \ar[r] & M \oplus_{\Omega^1_A} (A \otimes A) \ar[r] & A \ar[r] & 0
}
$$
Observe that this sequence splits if, and only if, $D$ is inner. Conversely, let
$$
\xymatrix{
0 \ar[r] & M \ar[r]^-i & E \ar[r]^-p & A \ar[r] & 0
}
$$
be a short exact sequence of $A^\ev$-modules. Since $p$ is surjective, $1 \in A$ has a preimage under $p$. Let $e \in E$ be such that $p(e) = 1$. Now the multiplication map $\mu_e: A \otimes A \rightarrow X$, $\mu_e(a \otimes b) = aeb$ is $A^\ev$-linear and such that the right square in
$$
\xymatrix{
0 \ar[r] & \Omega^1_A \ar[r]^-\iota \ar@{-->}[d]_-{\tilde{\mu}_e} & A \otimes A \ar[r]^-\mu \ar[d]^-{\mu_e} & A \ar@{=}[d] \ar[r] & 0 \\
0 \ar[r] & M \ar[r]^-i & E \ar[r]^-p & A \ar[r] & 0
}
$$
commutes. Hence the dashed arrow $\tilde{\mu}_e$ is induced as indicated, given by $\tilde{\mu}_e(\mathrm d a) = i^{-1}(a e - e a)$. The map $D_e = \tilde{\mu}_e \circ \mathrm d$ is a $K$-linear derivation; if $e'$ is another preimage of $1 \in A$ under $g$, the difference $D_e - D_{e'}$ is inner. The hereby obtained well-defined map $\Ext^1_{A^\ev}(A,M) \rightarrow \Out_K(A,M)$ is the inverse map of the isomorphism in (\ref{eq:surj}). If $e \in E$ is as above, we will call $D_e$ the derivation \textit{defined} by $e$.
\end{nn}

\begin{thm}\label{thm:gersteqmonoidal}
Let $M$ be an $A^\ev$-module. The following diagram commutes for $n = 0, 1$.
$$
\xymatrix{
\HH^n(A,M) \times Z_M(A) \ar[r]^-{[-,-]} \ar[d]_-{\cong} & \HH^{n-1}(A,M) \ar[d]^-{\cong}\\
\Ext^n_{A^\ev}(A,M) \times Z_{M}(\Mod(A^\ev)) \ar[r]^-{\langle -, -\rangle} & \Ext^{n-1}_{A^\ev}(A,M)
}
$$
\end{thm}

\begin{proof}
The statement is a triviality for $n = 0$, so we may assume that $n = 1$. Let 
$$
\xymatrix@C=18pt{
S & \equiv & 0 \ar[r] & M \ar[r]^-i & E \ar[r]^-p & A \ar[r] & 0
}
$$
be an exact sequence, $f \in Z_M(\Mod(A^\ev)) \subseteq \Hom_{A^\ev}(A,A)$ be a homomorphism, and let $z = f(1)$ be the element in $Z_M(A)$ that corresponds to $f$. Further, let $\Delta(f) = F_\lambda - F_\varrho$ be the difference of the morphisms defining the loop $\Omega_{A}(S,f)$. In light of Lemma \ref{lem:preimage_mu}, we have to show the following: For a given lifting $\Phi: \mathbb PA \rightarrow (S\#f)^\natural$ of $\id_A$, there is a null-homotopy $s_i$ (for $i \geqslant 0$) for the map $\Delta(f) \circ \Phi$ with $s_0(1 \otimes 1) = [D_S, z]_M = D_S(z)$, where $D_S$ denotes the (equivalence class of a) derivation in
$$
\xymatrix@C=18pt{
\Der_K(A,M) \ar@{->>}[r] & \displaystyle \frac{\Der_K(A,M)}{\Inn_K(A,M)} = \HH^1(A,M)
}
$$
that corresponds to $S$. Chose a lifting $\Phi: \mathbb PA \rightarrow (S \# f)^\natural$,
$$
\xymatrix{
\mathbb PA \ar[d] & \cdots \ar[r] & P_1 \ar[r]^-{\pi_1} \ar[d]_-{\varphi_1} & A \otimes A \ar[d]^-{\varphi_0} \ar[r]^-{\mu} & A \ar[r] \ar@{=}[d] & 0\,\,\\
S \# f & 0 \ar[r] & M \ar[r] & M \oplus_M E \ar[r] & A \ar[r] & 0 \, ,
}
$$
of the identity map of $A$. Evidently, $\varepsilon = \varphi_0(1 \otimes 1)$ is being mapped to $1$ by $p$, and $\varphi_0 = \mu_\varepsilon$. Now, the homomorphism $\Delta(f)$ of complexes is non-trivial in a single degree, namely in degree $0$, and the assignment $(m,e) \mapsto (0, f(1)e - ef(1))$ defines a map $M \oplus E \rightarrow A \oplus E$ which makes the diagram
$$
\xymatrix@C=30pt{
M \oplus_M E \ar[r]^-{\Delta(f)_0} & M \times_A A \ar@{ >->}[d]^-{\mathrm{can}}\\
M \oplus E \ar[r] \ar[r]^-{} \ar@{->>}[u]^-{\mathrm{can}} & A \oplus E
}
$$
commutative. We arrive at the following commutative diagram.
$$
\xymatrix{
\mathbb PA \ar[d]_-\Phi & \cdots \ar[r] & P_1 \ar[r]^-{\pi_1} \ar[d]_-{\varphi_1} & A \otimes A \ar[d]^-{\varphi_0} \ar[r]^-{\mu} & A \ar[r] \ar@{=}[d] & 0\\
S \# f \ar[d]_-{\Delta(f)} & 0 \ar[r] & M \ar[r] \ar[d]_-{0} & M \oplus_M E \ar[r] \ar[d]^-{\Delta(f)_0} & A \ar[r] \ar[d]^-{0} & 0\\
f \# S & 0 \ar[r] & M \ar[r]^-{j = \left[\begin{smallmatrix}
0\\
i
\end{smallmatrix}\right]
} & A \times_A E \ar[r] & A \ar[r] & 0
}
$$
Since the map $A \times_A E \rightarrow A$ sends $(0, f(1)\varepsilon - \varepsilon f(1))$ to zero, there is a unique preimage $m_\varepsilon$ of $(0, f(1)\varepsilon - \varepsilon f(1))$ under $j$. Indeed, $m_\varepsilon$ is given by $i^{-1}(f(1)\varepsilon - \varepsilon f(1))$. Now let
$$
s_0 = \mu_{m_\varepsilon} : A \otimes A \longrightarrow M, \ s_0(a \otimes b) = a m_\varepsilon b .
$$
Obviously, $s_0(\Omega_A^1) = 0$ so that $s_0 \circ \pi_1 = 0$. On the other hand,
\begin{align*}
(\Delta(f)_0 \circ \varphi_0)(a \otimes b) &= a(0, f(1)\varepsilon - \varepsilon f(1))b = a j(m_\varepsilon) b = (j \circ s_0)(a \otimes b)
\end{align*}
for all $a, b \in A$. Therefore $s_0$ defines a null-homotopy. Finally, under the isomorphism
$$
\Hom_{A^\ev}(A \otimes A, M) \longrightarrow \Hom_K(K, M), \ \varphi \mapsto \varphi(1 \otimes 1),
$$
$s_0$ is being mapped to $s_0(1 \otimes 1) = m_\varepsilon = j^{-1}(0,\varepsilon f(1) - f(1)\varepsilon) = i^{-1}(\varepsilon f(1) - f(1)\varepsilon)$ which precisely is the evaluation at $z = f(1)$ of the derivation $D_\varepsilon : A \rightarrow M$ defined by $\varepsilon$. To finish, recall that $D_\varepsilon$ is a representative of the equivalence class $D_S$ in $\Out_K(A,M)$.
\end{proof}

When combined with Lemma \ref{lem:EnC}, the theorem immediately yields the following.

\begin{cor}\label{cor:Outdervanish}
The map
$$
[-,-] : \frac{\Der_K(A,M)}{\Inn_K(A,M)} \times Z_M(A) \longrightarrow Z_M(A), \ (D,z) \mapsto D(z),
$$
is trivial if, and only if, for each extension $0 \rightarrow M \rightarrow E \rightarrow A \rightarrow 0$ of bimodules, one has $Z_M(A) \subseteq Z_E(A)$. In particular, $\{-,-\} : \Out_K(A, A) \times Z(A) \rightarrow Z(A)$ is trivial if, and only if, $Z_E(A) = Z(A)$ for every extension $0 \rightarrow A \rightarrow E \rightarrow A \rightarrow 0$.\qed
\end{cor}

\begin{rem}\label{rem:kerder}
By having a closer look at the results, and their proofs, leading to Corollary \ref{cor:Outdervanish}, we can improve its statement slightly, in the following sense: Let $D: A \rightarrow M$ be a derivation and let $z \in Z_M(A)$ be an element in the $M$-relative center of $A$. Then $D(z) = 0$ if, and only if, for the short exact sequence $0 \rightarrow M \rightarrow E_D \rightarrow A \rightarrow 0$ that corresponds to the equivalence class of $D$ under the isomorphism $\HH^1(A,M) = \Out_K(A,M) \xrightarrow{\sim} \Ext^1_{A^{\ev}}(A,M)$, one has $z \in Z_{E_D}(A)$.
\end{rem}

\begin{thm}\label{thm:gersteqmonoidalproj}
Let $M$ be an $A^\ev$-module. If $A$ is projective as a $K$-module, then the following diagram commutes for $n > 1$.
$$
\xymatrix{
\HH^n(A,M) \times Z_M(A) \ar[r]^-{[-,-]} \ar[d]_-{\cong} & \HH^{n-1}(A,M) \ar[d]^-{\cong}\\
\Ext^n_{A^\ev}(A,M) \times Z_{M}(\Mod(A^\ev)) \ar[r]^-{\langle -, -\rangle} & \Ext^{n-1}_{A^\ev}(A,M)
}
$$
\end{thm}

\begin{proof}
Since $A$ is $K$-projective by assumption, the bar resolution $\mathbb B A$ is a projective resolution of $A$ over $A^\ev$. We will therefore assume that $\mathbb P A = \mathbb B A$. Fix an extension $S \in \mathcal Ext^n_{A^\ev}(A,M)$ with corresponding equivalence class $\alpha = [S]$ in $\Ext^n_{A^\ev}(A,M)$, and a map $f \in Z_M(\Mod(A^\ev)) \subseteq \Hom_{A^\ev}(A,A)$. Let $S$ be given as
$$
\xymatrix@C=18pt{
S & \equiv & 0 \ar[r] & M \ar[r]^-{d_{n}} & E_{n-1} \ar[r]^-{d_{n-1}} & \cdots \ar[r]^-{d_2} & E_1 \ar[r]^-{d_1} & E_0 \ar[r]^-{d_0} & A \ar[r] & 0
}
$$
and let $z = f(1)$ be the element in $Z_M(A)$ determined by $f$. Chose a lifting $\Phi: \mathbb B A \rightarrow S^\natural$ of the identity map of $A$, and put
$$
\psi_{i} = \varphi_i(1 \otimes - \otimes \cdots \otimes - \otimes 1) \quad \in \quad \Hom_K(A^{\otimes i}, M)
$$
for $i \geqslant 0$. Recall the classical result (cf. \cite[Sec.\,IV.9]{HiSt97}), that the equivalence class of the $n$-cocycle $\psi_n = \varphi_n(1 \otimes - \otimes \cdots \otimes - \otimes 1)$ in $\HH^n(A,M)$ will be mapped to $\alpha$ under the isomorphism $\HH^n(A,M) \xrightarrow{\sim} \Ext^n_{A^\ev}(A,M)$. The homomorphism $\Phi$ of complexes yields a lifting $\tilde{\Phi}: \mathbb BA \rightarrow (S\#f)^\natural$ of the identity map of $A$:
$$
\xymatrix@C=16pt{
\cdots \ar[r]^-{\beta_{n+1}} & A^{\otimes(n+2)} \ar[d]_-{\varphi_n} \ar[r]^-{\beta_{n}} & A^{\otimes(n+1)} \ar[r]^-{\beta_{n-1}} \ar[d]_-{\varphi_{n-1}} & \cdots \ar[r]^-{\beta_2} & A \otimes A \otimes A \ar[r]^-{\beta_1} \ar[d]^-{\varphi_1} & A \otimes A \ar[r]^-{\beta_0} \ar[d]^-{\varphi_0} & A \ar[r] \ar@{=}[d] & 0\\
0 \ar[r] & M \ar[d]_-{f_\lambda^M}^-{f_\varrho^M}|= \ar[r] & E_{n-1} \ar[d]_-{\mathrm{can}} \ar[r] & \cdots \ar[r] & E_1 \ar[r] \ar@{=}[d] & E_0 \ar[r] \ar@{=}[d] & A \ar[r] \ar@{=}[d] & 0\\
0 \ar[r] & M \ar[r]^-{\mathrm{can}} & M \oplus_M E_{n-1} \ar[r] & \cdots \ar[r] & E_1 \ar[r] & E_0 \ar[r] & A \ar[r] & 0
}
$$
Let $F_\lambda$ and $F_\varrho$ be the morphisms of extensions defining the loop $\Omega_A(S,f)$, and $\Delta(f) = F_\lambda - F_\varrho$ be their difference. The task will be, as in the proof of the preceding theorem, to find a null-homotopy $s_i: A^{\otimes(i+2)} \rightarrow E_{i+1}$ (for $i \geqslant 0$) for $\Delta(f) \circ \tilde{\Phi}$ such that the image of $s_{n-1}$ under the isomorphism
$$
\Hom_{A^\ev}(A^{\otimes(n+1)}, M) \longrightarrow \Hom_K(A^{\otimes(n-1)},M), \ \varphi \mapsto \varphi(1 \otimes - \otimes \cdots \otimes - \otimes 1).
$$
represents the element $[\alpha, z] \in \HH^{n-1}(A, M)$. We claim that the maps
\begin{align*}
s_i(a_0 \otimes \cdots \otimes a_{i+1}) &= a_0 \big( (\psi_{i+1} \bullet z)(a_1 \otimes \cdots \otimes a_{i})\big) a_{i+1}\\
&= \sum_{k=1}^{i + 1}(-1)^{k-1} \varphi_{i+1}(a_0 \otimes \dots \otimes a_{k-1} \otimes z \otimes a_{k} \otimes \cdots \otimes a_{i+1})
\end{align*}
define the null-homotopy that we are seeking for. In fact, it is already apparent from the definition, that
$$
a_1 \otimes \cdots \otimes a_{n-1} \mapsto s_{n-1}(1 \otimes a_1 \otimes \cdots \otimes a_{n-1} \otimes 1) = (\psi_n \bullet z) (a_1 \otimes \cdots \otimes a_{n-1})
$$
is a map that represents $[\alpha, z] \in \HH^{n-1}(A,M)$.

As a first observation, the composition $\Delta(f) \circ \tilde{\Phi}$ is zero in degrees $n - 1$ and $n$. Its non-trivial component maps are
\begin{align*}
(\Delta(f) \circ \tilde{\Phi})_0(a \otimes b) &= (0,z \varphi_0(a \otimes b) - \varphi_0(a \otimes b) z) && \text{(for $a, b \in A$)},\\
(\Delta(f) \circ \tilde{\Phi})_i &= (f_\lambda^{E_{i}} - f_\varrho^{E_i}) \circ \varphi_i && \text{(for $i \neq 0, n-1, n$)}.
\end{align*}
The map $(\Delta(f) \circ \tilde{\Phi})_0$ coincides with $(E_{1} \rightarrow A \times_A E_0) \circ s_0$, for the latter sends $a \otimes b \in A \otimes A$ to $(0, d_1(s_0(a \otimes b))$ and
\begin{align*}
(d_1 \circ s_0)(a \otimes b) &= (d_1 \circ \varphi_1)(a \otimes z \otimes b)\\
&= (\varphi_0 \circ \beta_1)(a \otimes z \otimes b)\\
&= \varphi_0(az \otimes b) - \varphi(a \otimes zb)\\
&= z\varphi_0(a \otimes b) - \varphi(a \otimes b)z .
\end{align*}
Finally, if $\mathbb C(A,M) = (C^\ast(A,M), \partial_M)$ denotes the Hochschild cocomplex with coefficients in $M$, we get
$$
\partial_M(\psi_{i+1} \bullet z) + \partial_M(\psi_{i+1}) \bullet z = \psi_{i+1} \bullet \partial_A(z) + z\psi_{i+1} - \psi_{i+1}z \, ,
$$
by the fundamental formula (\ref{eq:fundform}), which yields
\begin{align*}
s_i \circ \beta_{i+1} + d_{i+2} \circ s_{i+1} &= z\varphi_{i+1} - \varphi_{i+1} z\\ &= \Delta(f)_{i+1} \circ \varphi_{i+1}.
\end{align*}
In fact, one easily verfies that
\begin{align*}
(s_i \circ \beta_{i+1})(a_0 \otimes \cdots \otimes a_{i+2}) &= a_0 \big(\partial_M(\psi_{i+1} \bullet z)(a_1 \otimes \cdots \otimes a_{i+1})\big) a_{i+2}
\intertext{and}
(d_{i+2} \circ s_{i+1})(a_0 \otimes \cdots \otimes a_{i+2}) &= a_0 \big( (\partial_M (\psi_{i+1}) \bullet z)(a_1 \otimes \cdots \otimes a_{i+1})\big)a_{i+2}
\end{align*}
for all $i = 0, \dots, n$ and all $a_0, \dots, a_{i+2} \in A$.
\end{proof}

From Example \ref{exa:monoidalfunc} and Proposition \ref{prop:functcomp} we immediately deduce the following.

\begin{cor}\label{cor:mapMoritainv}
Let $A$ be projective as a $K$-module. Then, if $B$ is a $K$-algebra being Morita equivalent to $A$, with corresponding progenerator $P$ for $A$, there is an isomorphism $\HH^\ast(A,M) \cong \HH^\ast(B, N)$ of right Gerstenhaber modules over $Z_M(A) \cong Z_{N}(B)$. Here $N$ denotes the $B^\ev$-module $\Hom_{A^\ev}(P \otimes \Hom_A(P,A), M)$.\qed
\end{cor}

\begin{cor}\label{cor:classcomm}
Let $A$ be projective over $K$. Consider the following statements.
\begin{enumerate}[\rm(1)]
\item\label{cor:classcomm:1} The Gerstenhaber bracket $\{-,-\}$ on $\HH^\ast(A)$ introduced in Section $\ref{sec:Mrel}$ is trivial.
\item\label{cor:classcomm:2} The category $(\Mod(A^\ev), \otimes_{A}, A)$ is braided monoidal.
\item\label{cor:classcomm:3} The category $(\Mod(Z(A)^\ev), \otimes_{Z(A)}, Z(A))$ is braided monoidal.
\item\label{cor:classcomm:4} $Z(A) = Z_M(A)$ for all $M \in \Mod(A^\ev)$.
\item\label{cor:classcomm:5} The bracket $[-,-]_M : \HH^\ast(A,M) \times Z_M(A) \rightarrow \HH^{\ast - 1}(A,M)$ vanishes for all $M \in \Mod(A^\ev)$.
\end{enumerate}
Then one has the implications
$$
(\ref{cor:classcomm:1}) \ \Longleftarrow \ (\ref{cor:classcomm:2}) \ \Longrightarrow \ (\ref{cor:classcomm:3}) \ \Longleftrightarrow \ (\ref{cor:classcomm:4}) \ \Longrightarrow \ (\ref{cor:classcomm:5})
$$
amongst them.
\end{cor}

\begin{proof}
Due to Lemma \ref{lem:Mrelcent} and Proposition \ref{prop:Xrelcennat} the only implications that remain to be shown are $(\ref{cor:classcomm:1}) \, \Longleftarrow \, (\ref{cor:classcomm:2}) \, \Longrightarrow \, (\ref{cor:classcomm:3})$. If $(\Mod(A^\ev), \otimes_{A}, A)$ is braided monoidal, with braiding $\gamma_{M,N}: M \otimes_A N \rightarrow N \otimes_A M$, then the axioms yield that $\gamma_{A,N} = \varrho^{-1}_N \circ \lambda_{N} = \gamma_N$. In particular, $\gamma_N$ will give rise to an isomorphism
$$
( - \otimes_A - ) \longrightarrow (- \otimes_A -) \circ T
$$
of functors $\add_{A^\ev}(A) \times \Mod(A^\ev) \rightarrow \Mod(A^\ev)$. Therefore, by Proposition \ref{prop:Xrelcennat}, the second item implies the forth, hence the third.

Moreover, item (\ref{cor:classcomm:2}) implies, that the Gerstenhaber bracket $\{-,-\}$ vanishes in degrees different from $(n,0)$; see \cite[Cor.\,5.5.8]{He14b}. But since it also implies (\ref{cor:classcomm:3}), and hence (\ref{cor:classcomm:5}), the bracket will vanish in degrees $(n,0)$ as well.
\end{proof}

The implications $(\ref{cor:classcomm:1}) \, \Longrightarrow \, (\ref{cor:classcomm:2})$ and $(\ref{cor:classcomm:1}) \, \Longrightarrow \, (\ref{cor:classcomm:3})$ in the above corollary will, in general, not hold true. Thus the question is:

\begin{quest}
What does it mean for the algebra $A$ and its category of (bi-)modules that the restricted Gerstenhaber bracket,
$$
\{-,-\} : \HH^m(A) \times \HH^n(A) \longrightarrow \HH^{m + n - 1}(A),
$$
vanishes for all, or some, integers $m, n \geqslant 0$?
\end{quest}

\bigskip\bigskip
\noindent\textbf{Acknowledgements.} I thank R.-O.\,Buchweitz, S.\,Oppermann and J.\,Steen for valuable discussions on the topic. The research that led to this paper was supported by the project ``Triangulated categories in algebra'' (Norwegian Research Council project NFR 221893).

\bibliographystyle{abbrv}

%\printindex

\end{document}